\documentclass[11pt]{amsart}
\usepackage[T1]{fontenc}
\usepackage[latin1]{inputenc}
\usepackage{a4}
\usepackage{setspace}
\usepackage{xypic}
\usepackage{amssymb}
\usepackage{amsthm}
\usepackage[english]{babel}
\usepackage{latexsym}
\date{}

\usepackage{srcltx}
\newtheorem{thm}{Theorem}[section]
\newtheorem*{thm*}{Theorem}

\newtheorem{defn}[thm]{Definition}

\newtheorem{Hypothese}[thm]{Hypothesis}

\newtheorem{rem}[thm]{Remark}
\newtheorem{prop}[thm]{Proposition}
\newtheorem{lem}[thm]{Lemma}
\newtheorem{cor}[thm]{Corollary}

\newcommand{\fonc}[5]{
 \begin{array}{cccc}
 #1: & #2 & \longrightarrow & #3\\
     & #4 & \longmapsto & #5
 \end{array}
}

\newcommand{\appl}[4]{
 \begin{array}{cccc}
   #1 & \longrightarrow & #2\\
   #3 & \longmapsto & #4
 \end{array}
}

\begin{document}

\title[The Weil-étale fundamental group I]{The Weil-étale fundamental group of a number field I}
\author{Baptiste Morin}
\address{Department of Mathematics, Caltech, Pasadena CA 91125, USA}

\maketitle

\begin{abstract}
Lichtenbaum has conjectured in \cite{Lichtenbaum} the existence of a Grothendieck topology for an arithmetic scheme $X$ such that the Euler characteristic of the cohomology groups of the constant sheaf $\mathbb{Z}$ with compact support at infinity gives, up to sign, the leading term of the zeta-function $\zeta_X(s)$ at $s=0$. In this paper we consider the category of sheaves $\bar{X}_L$ on this conjectural site for $X=Spec(\mathcal{O}_F)$ the spectrum of a number ring. We show that $\bar{X}_L$ has, under natural topological assumptions, a well defined fundamental group whose abelianization is isomorphic, as a topological group, to the Arakelov Picard group of $F$. This leads us to give a list of topological properties that should be satisfied by $\bar{X}_L$. These properties can be seen as a global version of the axioms for the Weil group. Finally, we show that any topos satisfying these properties gives rise to complexes of \'etale sheaves computing the expected Lichtenbaum cohomology.
\end{abstract}


\section{Introduction}
Lichtenbaum has conjectured in \cite{Lichtenbaum} the existence of a Grothendieck topology for an arithmetic scheme $X$ such that the Euler characteristic of the cohomology groups of the constant sheaf $\mathbb{Z}$ with compact support at infinity gives, up to sign, the leading term of the zeta-function $\zeta_X(s)$ at $s=0$. There should exist motivic complexes of sheaves $\mathbb{Z}(n)$ giving the special value of $\zeta_X(s)$ at any non-positive integer $s=n$, and this formalism should extend to motivic $L$-functions. In this paper, this conjectural cohomology theory will be called the \emph{conjectural Lichtenbaum cohomology}. This cohomology is well defined for schemes of finite type over finite fields, by the work of Lichtenbaum and Geisser. But the situation for flat schemes over $\mathbb{Z}$ is more difficult, and far from being understood even in the most simple case $X=Spec(\mathbb{Z})$.

We denote by $\bar{X}$ the Arakelov compactification of a number ring $X=Spec(\mathcal{O}_F)$. Using the Weil group $W_F$, Lichtenbaum has defined a first candidate for his conjectural cohomology of number rings, which he calls the Weil-étale cohomology. He has shown that the resulting cohomology groups with compact support $H_{Wc}^i(X,\mathbb{Z})$, assuming that they vanish for $i\geq3$, are indeed miraculously related to the special value of the Dedekind zeta function $\zeta_F(s)$ at $s=0$. However, Matthias Flach has shown in \cite{MatFlach} that the groups $H_{W}^i(\bar{X},\mathbb{Z})$ (hence $H_{Wc}^i(X,\mathbb{Z})$) are in fact infinitely generated for any $i\geq4$ even. This shows that Lichtenbaum's definition is not yet the right one, as it is mentioned in \cite{Lichtenbaum}. Giving the correct (site of) definition for the conjectural Lichtenbaum cohomology of number rings is a deep problem.

This problem can not be attacked directly. One has first to figure out the basic properties that need to be satisfied by Lichtenbaum's conjectural site. This question only makes sense if one considers the category of sheaves of sets on Lichtenbaum's conjectural site, i.e. the associated topos, since many non equivalent sites can produce the same topos. In this paper, this conjectural topos will be denoted by $\bar{X}_{L}$ and will be called the \emph{conjectural Lichtenbaum topos}. The first goal of this paper is to figure out the basic topological properties that must be satisfied by the conjectural Lichtenbaum topos of a number ring. To this aim, we give in section \ref{subsect-expected-properties} a list of 9 necessary axioms for $\bar{X}_{L}$. Our axioms globalize the usual axioms for the Weil group. From now on, we refer to them as \emph{Axioms $(1)-(9)$}. They are all satisfied by the (naturally defined) Weil-\'etale topos of a smooth projective curve over a finite field, and they are in adequation with the work of Deninger (see \cite{Deninger}).

In Sections 4 and 5, we show that Axioms $(1)-(9)$ must be satisfied by $\bar{X}_{L}$ under natural topological assumptions. Our argument is based on the following observation. The cohomology of the topos $\bar{X}_{L}$ associated to the arithmetic curve $\bar{X}=\overline{Spec(\mathcal{O}_F)}$, with coefficients in  $\mathbb{Z}$ and $\widetilde{\mathbb{R}}$, must be the same as the one computed in \cite{Lichtenbaum} in degrees $i\leq3$ and must vanish in degrees $i\geq4$. This is necessary in order to obtain the correct cohomological interpretation (due to Lichtenbaum) of the analytic class number formula. Then we show that such a topos has, under natural topological assumptions, a well defined fundamental group whose abelianization is isomorphic, \emph{as a topological group}, to the Arakelov Picard group $Pic(\bar{X})$ (see Theorem \ref{main-thm}). In order to prove this result, we express Pontryagin duality in terms of sheaves and we use the notion of topological fundamental groups (as a special case of the fundamental group of a connected and locally connected topos over an arbitrary base topos with a point). Moreover, this argument also applies to the case of an arbitrary connected \'etale $\bar{X}$-scheme $\bar{U}$. Here we find that the abelian fundamental group of the slice topos $\bar{X}_{L}/\bar{U}$ must be topologically isomorphic to the $S$-id\`ele class group canonically associated to $\bar{U}$.

Axioms $(1)-(9)$ give a partial description of the conjectural Lichtenbaum topos. This description, which can be seen as a global version of the axioms for the Weil group (see \cite{Tate}), is based on an interpretation of the $S$-idele class groups in terms of topological fundamental groups. Section \ref{Sect-Cohomology} is devoted to the proof of the following result.
\begin{thm}\label{thm-intro-L-formalism}\textbf{\emph{(Lichtenbaum's formalism)}}
Assume that $F$ is totally imaginary. Let $\gamma:\bar{X}_{L}\rightarrow\bar{X}_{et}$ be any topos satisfying Axioms $(1)-(9)$. We denote by $\varphi:X_{L}\rightarrow\bar{X}_{L}$ the natural open embedding, and we set $H_c^n(X_{L},\widetilde{\mathbb{R}}):=H^n(\bar{X}_{L},\varphi_!\widetilde{\mathbb{R}})$. The following is true:
\begin{itemize}
\item $\mathbb{H}^n(\bar{X}_{et},\tau_{\leq2}R\gamma_*(\varphi_!\mathbb{Z}))$ is finitely generated, zero for $n\geq4$ and
the canonical map $$\mathbb{H}^n(\bar{X}_{et},\tau_{\leq2}R\gamma_*(\varphi_!\mathbb{Z}))\otimes\mathbb{R}\longrightarrow H_c^n({X}_{L},\widetilde{\mathbb{R}})$$
is an isomorphism for any $n\geq0$.
\item There exists a fundamental class $\theta\in{H}^1(\bar{X}_{L},\widetilde{\mathbb{R}})$. The complex of finite dimensional vector spaces
    $$...\rightarrow{H}_c^{n-1}({X}_{L},\widetilde{\mathbb{R}})\rightarrow
 {H}_c^n({X}_{L},\widetilde{\mathbb{R}})\rightarrow
 {H}_c^{n+1}({X}_{L},\widetilde{\mathbb{R}})\rightarrow...$$
    defined by cup product with $\theta$, is acyclic.
\item Let $B^n$ be a basis of $\mathbb{H}^n({X}_{et},\tau_{\leq2}R\gamma_*(\varphi_!\mathbb{Z}))/\textsl{tors}$. The leading term coefficient $\zeta^*_F(0)$ at $s=0$ is given by the Lichtenbaum Euler characteristic :
$$\zeta^*_F(0)=\pm\prod_{n\geq0}|\mathbb{H}^n(\bar{X}_{et},\tau_{\leq2}R\gamma_*(\varphi_!\mathbb{Z}))_{\textsl{tors}}|^{(-1)^n}/
\textsl{det}(H_c^n({X}_{L},\widetilde{\mathbb{R}}),\theta,B^*)$$
\end{itemize}
\end{thm}
In \cite{Fundamental-group-II}, we construct a topos (the Weil-étale topos) which satisfies Axioms $(1)-(9)$.\\

$\mathbf{Acknowledgments}.$ I am very grateful to Matthias
Flach, Masanori Morishita and Fr\'ed\'eric Paugam for their comments.

\section{Preliminaries}

\subsection{Basic properties of geometric morphisms}

Let $\mathcal{S}$ and $\mathcal{S}'$ be two Grothendieck topoi. A \emph{(geometric) morphism of topoi}
$$f:=(f^*,f_*):\mathcal{S}'\longrightarrow\mathcal{S}$$
is defined as a pair of functors $(f^*,f_*)$, where $f^*:\mathcal{S}\rightarrow \mathcal{S}'$ is left adjoint to $f_*:\mathcal{S}'\rightarrow \mathcal{S}$ and $f^*$ is left exact (i.e. $f^*$ commutes with finite projective limits). One can also define such a morphism as a left exact functor $f^*:\mathcal{S}\rightarrow \mathcal{S}'$ commuting with arbitrary inductive limits. Indeed, in this case $f^*$ has a uniquely determined right adjoint $f_*$.

If $X$ is an object of $\mathcal{S}$, then the slice category $\mathcal{S}/X$, of objects of $\mathcal{S}$ over $X$, is a topos as well. The base change functor
$$\appl{\mathcal{S}}{\mathcal{S}/X}{Y}{Y\times X}$$
is left exact and commutes with arbitrary inductive limits, since inductive limits are universal in a topos. We obtain a morphism
$$\mathcal{S}/X\longrightarrow \mathcal{S}.$$
Such a morphism is said to be a \emph{localization morphism} or a \emph{local homeomorphism} (the term local homeomorphism is inspired by the case where $\mathcal{S}$ is the topos of sheaves on some topological space). For any morphism $f:\mathcal{S}'\rightarrow\mathcal{S}$ and any object $X$ of $\mathcal{S}$, there is a natural morphism
$$f_{/X}:\mathcal{S}'/f^*X\longrightarrow\mathcal{S}/X.$$
The functor $f_{/X}^*$ is defined in the obvious way: $f^*_{/X}(Y\rightarrow X)=(f^*Y\rightarrow f^*X)$. The direct image functor $f_{/X,*}$ sends $Z\rightarrow f^*X$ to $f_*Z\times_{f_*f^*X}X\rightarrow X$, where $X\rightarrow f_*f^*X$ is the adjunction map. The morphism $f_{/X}$ is a pull-back of $f$, in the sense that the square
\[\xymatrix{
\mathcal{S}'/f^*X\ar[r]^{f_{/X}}\ar[d]&\mathcal{S}/X\ar[d]\\
\mathcal{S}'\ar[r]^{f}&\mathcal{S}
}\]
is commutative and 2-cartesian. In other words, the 2-fiber product $\mathcal{S}'\times_{\mathcal{S}}\mathcal{S}/X$ can be defined as the slice topos $\mathcal{S}'/f^*X$.

A morphism $f:\mathcal{S}'\rightarrow\mathcal{S}$ is said to be \emph{connected} if $f^*$ is fully faithful. It is \emph{locally connected} if $f^*$ has an $\mathcal{S}$-indexed left adjoint $f_!$ (see \cite{elephant}
C3.3). These definitions generalize the usual ones for topological spaces: if $T$ is a topological space, consider the unique morphism $Sh(T)\rightarrow\underline{Sets}$ where $Sh(T)$ is the category of \'etal\'e spaces over $T$. For example a localization morphism $\mathcal{S}/X\rightarrow \mathcal{S}$ is always locally connected (here $f_!(Y\rightarrow X)=Y$), but is connected if and only if $X$ is the final object of $\mathcal{S}$.

A morphism $f:\mathcal{S}'\rightarrow\mathcal{S}$ is said to be an \emph{embedding} when $f_*$ is fully faithful. It is an \emph{open embedding} if $f$ factor through $f:\mathcal{S}'\simeq\mathcal{S}/X\rightarrow\mathcal{S}$, where $X$ is a subobject of the final object of $\mathcal{S}$. Then the essential image $\mathcal{U}$ of the functor $f_*$ is said to be an \emph{open subtopos} of $\mathcal{S}$. The \emph{closed complement} $\mathcal{F}$ of $\mathcal{U}$ is the strictly full subcategory of $\mathcal{S}$ consisting in objects $Y$ such that $Y\times X$ is the final object of $\mathcal{U}$ (i.e. $f^*Y$ is the final object of $\mathcal{S}'$). A \emph{closed subtopos} $\mathcal{F}$ of $\mathcal{S}$ is a strictly full subcategory which is the closed complement of an open subtopos. A morphism of topoi $i:\mathcal{E}\rightarrow\mathcal{S}$ is said to be a \emph{closed embedding} if $i$ factors through $i:\mathcal{E}\simeq\mathcal{F}\rightarrow\mathcal{S}$ where $\mathcal{F}$ is a closed subtopos of $\mathcal{S}$.

A \emph{subtopos} of $\mathcal{S}$ is a strictly full subcategory $\mathcal{S}'$ of $\mathcal{S}$ such that the inclusion functor $i:\mathcal{S}'\hookrightarrow\mathcal{S}$ is the direct image of a morphism of topoi (i.e. $i$ has a left exact left adjoint). A morphism $f:\mathcal{S}'\rightarrow\mathcal{S}$ is said to be \emph{surjective} if $f^*$ is faithful. Any morphism $f:\mathcal{E}\rightarrow\mathcal{S}$ can be decomposed as a surjection $\mathcal{E}\rightarrow Im(f)$ followed by a an embedding $Im(f)\rightarrow\mathcal{S}$, where $Im(f)$ is a subtopos of $\mathcal{S}$, which is called the \emph{image of $f$} (see \cite{SGA4} IV. 9.1.7.2).
\subsection{The topos $\mathcal{T}$ of locally compact topological spaces}

We denote by $Top$ the category of locally compact Hausdorff topological spaces and continuous maps. This category is endowed with the open cover topology $\mathcal{J}_{op}$, which is generated by the following pretopology: a family of morphisms $(X_{\alpha}\rightarrow X)_{\alpha\in A}$ is in $Cov(X)$ if and only if $(X_{\alpha}\rightarrow X)_{\alpha\in A}$ is an open cover in the usual sense. We denote by $\mathcal{T}$ the topos of sheaves of sets on this left exact site:
$$\mathcal{T}:=\widetilde{(Top,\mathcal{J}_{op})}.$$
The family of compact spaces is easily seen to be a topologically generating family for the site $(Top,\mathcal{J}_{op})$. Indeed, if $X$ is a locally compact space, then any $x\in X$ has a compact neighborhood $K_x\subseteq X$, so $(K_x\hookrightarrow X)_{x\in X}$ is a local section cover, hence a covering family for $\mathcal{J}_{op}$. In particular, if we denote by $Top^c$ the category of compact spaces, then $\mathcal{T}=\widetilde{(Top^c,\mathcal{J}_{op})}$. The Yoneda functor
$$\fonc{y}{Top}{\mathcal{T}}{X}{y(X)=Hom_{Top}(-,X)},$$
which sends a space $X$ to the sheaf represented by $X$,
is fully faithful (since $\mathcal{J}_{op}$ is subcanonical) and
commutes with arbitrary projective limits. Since the Yoneda functor is left exact, any locally compact topological group $G$ represents a group object of $\mathcal{T}$. In what follows we consider
$Top$ as a (left exact) full subcategory of $\mathcal{T}$. For
example the sheaf of $\mathcal{T}$ represented by a (locally compact Hausdorff) space $Z$
is sometimes also denoted by $Z$.

In this paper, we consider topoi defined over the topos of locally compact spaces since all sheaves, cohomology groups and fundamental groups that we use are defined by locally compact spaces. In order to use non-locally compact coefficients,
one can consider the topos $$\mathcal{T}':=\widetilde{(Top^h,\mathcal{J}_{op})}$$
where $Top^h$ is the category of Hausdorff spaces. Then for any topos $\mathcal{E}$ (connected and locally connected) over $\mathcal{T}$, we consider the base change $\mathcal{E}\times_{\mathcal{T}}\mathcal{T}'$ to obtain a (connected and locally connected) topos over $\mathcal{T}'$.

\subsection{The classifying topos of a group object}
For any topos $\mathcal{S}$ and any group object $\mathcal{G}$ in $\mathcal{S}$, we denote by $B_{\mathcal{G}}$ the category of (left) $\mathcal{G}$-object in $\mathcal{S}$. Then $B_{\mathcal{G}}$ is a topos, as it follows from Giraud's axioms, and $B_{\mathcal{G}}$ is endowed with a canonical morphism $B_{\mathcal{G}}\rightarrow\mathcal{S}$, whose inverse image functor sends an object $\mathcal{F}$ of $\mathcal{S}$ to $\mathcal{F}$ with trivial $\mathcal{G}$-action. If there is a risk of ambiguity, we denote the topos $B_{\mathcal{G}}$ by $B_{\mathcal{S}}(\mathcal{G})$. The topos $B_{\mathcal{G}}$ is said to be the classifying topos of $\mathcal{G}$ since it classifies $\mathcal{G}$-torsors. More precisely, for any topos $f:\mathcal{E}\rightarrow\mathcal{S}$ over $\mathcal{S}$, the category ${\underline{Homtop}}_{\mathcal{S}}\,(\mathcal{E},B_{\mathcal{G}})$ is equivalent to the category of $f^*\mathcal{G}$-torsors in $\mathcal{E}$ (see \cite{SGA4} IV Exercise 5.9).

\subsubsection{The classifying topos of a profinite group.}

Let $G$ be a discrete group, i.e. a group object of the final topos $\underline{Sets}$. We denote the category of $G$-sets by
$$B_G^{sm}:=B_{\underline{Sets}}(G).$$ The topos $B_G^{sm}$ is called the small classifying topos of the discrete group $G$.

If $G$ is a profinite group, then the small classifying topos $B_G^{sm}$ is defined as the category of sets on which $G$ acts continuously.

\subsubsection{The classifying topos of a topological group.}
Let $G$ be a locally compact (hence Hausdorff) topological group. Then $G$ represents a group object of $\mathcal{T}$, where $\mathcal{T}:=\widetilde{(Top,\mathcal{J}_{op})}$ is defined above. Then $$B_G:=B_{\mathcal{T}}(G)$$ is the classifying topos of the locally compact topological group $G$. One can define the classifying topos of an arbitrary topological group by enlarging the topos $\mathcal{T}$.

\subsubsection{Topological pro-groups}
In this paper, a \emph{filtered category} $I$ is a non-empty small category such that the following holds. For any objects $i$ and $j$ of $I$, there exists a third object $k$ and maps $i\leftarrow k\rightarrow j$. For any pair of maps $i\rightrightarrows j$ there exists a map $k\rightarrow i$ such that the diagram $k\rightarrow i\rightrightarrows j$ is commutative. Let $C$ be any category. A \emph{pro-object} of $C$ is a functor $\underline{X}:I\rightarrow C$, where $I$ is a filtered category. One can see a pro-object in $C$ as a diagram in $C$.
One can define the \emph{category $Pro(C)$ of pro-objects in $C$} (see \cite{SGA4} I. 8.10). The morphisms in this category can be made explicit as follows. Let $\underline{X}:I\rightarrow C$ and $\underline{Y}:J\rightarrow C$ be two pro-objects in $C$. Then one has
$$Hom_{Pro(C)}(\underline{X},\underline{Y}):=\underleftarrow{lim}_{_{j\in J}}\,\underrightarrow{lim}_{_{i\in I}}\, Hom(X_i,Y_j).$$
A pro-object $\underline{X}:I\rightarrow C$ is \emph{constant} if it is a constant functor, and $\underline{X}:I\rightarrow C$ is \emph{essentially constant} if $\underline{X}$ is isomorphic (in the category $Pro(C)$) to a constant pro-object.

\begin{defn}\label{defn-progroup}
A \emph{locally compact topological pro-group} $\underline{G}$ is a pro-object in the category of locally compact and Hausdorff topological groups. A locally compact topological pro-group is said to be \emph{strict} if any transition map $G_j\rightarrow G_i$ has local sections.
\end{defn}
If the category $C$ is a topos, then a pro-object $\underline{X}:I\rightarrow C$ in $C$ is said to be \emph{strict} when the transition map $X_i\rightarrow X_j$ is an epimorphism in $C$, for any $i\rightarrow j\in Fl(I)$. In particular, a locally compact topological pro-group $\underline{G}:I\rightarrow Gr(Top)$ pro-represents a strict pro-group-object in $\mathcal{T}$:
$$y\circ\underline{G}:I\rightarrow Gr(Top)\rightarrow Gr(\mathcal{T})$$
where $Gr(Top)$ and  $Gr(\mathcal{T})$ are the categories of group-objects in $Top$ and $\mathcal{T}$ respectively. Indeed, a continuous map of locally compact spaces $X_i\rightarrow X_j$ has local sections if and only if it induces an epimorphism $y(X_i)\rightarrow y(X_j)$ in $\mathcal{T}$. Topos theory provides a natural way to define the limit of a strict topological pro-group without any loss of information.
\begin{defn}\label{defn-classifying-topos-pro-grp}
The classifying topos of a locally compact strict pro-group $\underline{G}:I\rightarrow Gr(Top)$ is defined as
$$B_{\underline{G}}:=\underleftarrow{lim}_{I}\, B_{G_i},$$
where the the projective limit is computed in the 2-category of topoi.
\end{defn}

\subsection{The Arakelov Picard group}

Let $F$ be a number field and let $X_{\infty}$ be the set of
archimedean places of $F$. We denote by
$\bar{X}=(Spec\,\mathcal{O}_F,X_{\infty})$ the Arakelov
compactification of the ring of integers in $F$. We consider the
idèle group $I_F$ and the idèle class group $C_F$ of $F$. For any
place $v$ of $F$, we denote by $F_v$ the corresponding local field
and by $\mathcal{O}^{\times}_{F_v}$ the group of local units, i.e.
the kernel of the absolute value
$K^{\times}_v\rightarrow\mathbb{R}_{>0}$. Note that we have
$\mathcal{O}^{\times}_{F_v}=\mathbb{S}^1$ for $v$ complex and $\mathcal{O}^{\times}_{F_v}=\{\pm1\}$ for $v$ real. The group
$\mathcal{O}^{\times}_{F_v}$ is always compact. The \emph{Arakelov Picard
group} $Pic(\bar{X})$ is defined as the cokernel, endowed with the
quotient topology, of the continuous map $\prod_{v}\mathcal{O}^{\times}_{F_v}\rightarrow C_F$.
For any place $v$, the map $F_v^{\times}\rightarrow C_F$ induces a
continuous morphism
\begin{equation}\label{map-residue-Weil-group-to-Pic}
W_{k(v)}:=F_v^{\times}/\mathcal{O}^{\times}_{F_v}\longrightarrow
Pic(\bar{X})
\end{equation}
where $W_{k(v)}$ is the Weil group of the "residue field $k(v)$" at
$v\in\bar{X}$. The absolute value $C_F
\rightarrow\mathbb{R}_{>0}$ factors through $Pic(\bar{X})$. We obtain a
canonical continuous morphism $Pic(\bar{X})\rightarrow\mathbb{R}_{>0}$
endowed with a continuous section. The Arakelov class group
$Pic^1(\bar{X})$ is the kernel of this map. In other words, we have
an exact sequence of topological group
$$0\rightarrow Pic^1(\bar{X})\rightarrow Pic(\bar{X})\rightarrow\mathbb{R}_{>0}\rightarrow0.$$

\section{Pontryagin Duality and topological fundamental groups}

\subsection{Pontryagin Duality}

Let $X$ and $Y$ be two objects in a topos $\mathcal{E}$. There
exists an internal Hom-object $\underline{Hom}_{\mathcal{E}}(X,Y)$
in $\mathcal{E}$ such that there is a functorial isomorphism
\begin{equation}\label{hom-object}
Hom(Z,\underline{Hom}_{\mathcal{E}}(X,Y))\simeq Hom(Z\times
X,Y)=Hom_{\mathcal{E}/_Z}(Z\times X,Z\times Y)
\end{equation}
for any object $Z$ of $\mathcal{E}$. Indeed, the (base change)
functor $Z\rightarrow Z\times X$ commutes with (arbitrary) inductive
limits since inductive limits are universal in the topos
$\mathcal{E}$. Therefore, the contravariant functor
$$\appl{\mathcal{E}}{\underline{Set}}{Z}{Hom_{\mathcal{E}}(Z\times X,Y)}$$
sends inductive limits in $\mathcal{E}$ to projective limits in
$\underline{Set}$. Hence\underline{} this presheaf on $\mathcal{E}$
is a sheaf for the canonical topology. Since the sheaves on a topos
endowed with the canonical topology are all representable, this
functor is representable by an object
$\underline{Hom}_{\mathcal{E}}(X,Y)$ of $\mathcal{E}$. If $G$ and
$A$ are both group objects in $\mathcal{E}$ such that $A$ is abelian, then we denote by
$\underline{Hom}_{\mathcal{E}}(G,A)$ the group object of
$\mathcal{E}$ given by 
$$\appl{\mathcal{E}}{\underline{Ab}}{Z}{Hom_{Gr(\mathcal{E}/Z)}(Z\times G,Z\times A)},$$
where $\underline{Ab}$ and $Gr(\mathcal{E}/Z)$ denote respectively
the category of (discrete) abelian groups and the category of group objects
of the slice topos $\mathcal{E}/Z$.

Let $\mathcal{T}$ be the topos of sheaves on the site
$(Top,\mathcal{J}_{op})$, where $Top$ is the category
\emph{Hausdorff locally compact} topological spaces and continuous maps endowed with
the open cover topology $\mathcal{J}_{op}$. Recall that the Yoneda functor
$$\fonc{y}{Top}{\mathcal{T}}{X}{y(X)=Hom_{Top}(-,X)}$$
sending a topological space to the sheaf represented by this space
is fully faithful and
commutes with arbitrary projective limits.

Let $X$ and $Y$ be two Hausdorff locally compact topological spaces.
We denote by $\underline{Hom}_{Top}(X,Y)$ the set of continuous maps
from $X$ to $Y$ endowed with the compact-open topology. This
topological space is Hausdorff and locally compact. Then the sheaf
of $\mathcal{T}$ represented by $\underline{Hom}_{Top}(X,Y)$ is
precisely the internal object $\underline{Hom}_{\mathcal{T}}(y(X),y(Y))$
defined above, since $Hom_{Top}(X,Y)$ satisfies (\ref{hom-object}).
Indeed, one has
$$Hom_{Top}(Z\times X,Y)=Hom_{Top}(Z,\underline{Hom}_{Top}(X,Y))$$
for any Hausdorff topological spaces $Z$. Hence the sheaf
$\underline{Hom}_{\mathcal{T}}(y(X),y(Y))$ is represented by
$\underline{Hom}_{Top}(X,Y)$, i.e. one has a canonical isomorphism
in $\mathcal{T}$:
$$\underline{Hom}_{\mathcal{T}}(y(X),y(Y))=y(\underline{Hom}_{Top}(X,Y)).$$

If $G$ and $A$ are two Hausdorff locally compact topological groups such that $A$ is abelian
then the abelian group of continuous morphisms $\underline{Hom}_{Top}(G,A)$
is also endowed with the compact-open topology, and we have
$$\underline{Hom}_{\mathcal{T}}(y(G),y(A))=y(\underline{Hom}_{Top}(G,A)).$$
Note that $y(G)$ and $y(A)$ are two group objects in $\mathcal{T}$
since the Yoneda functor $y$ commutes with finite projective limits.

\begin{defn}\label{defn-dual}
Let $\mathcal{G}$ be a group object of $\mathcal{T}$. We denote by $\mathcal{G}^D$ the
internal Hom-group-object of $\mathcal{T}$:
$$\mathcal{G}^D:=\underline{Hom}_{\mathcal{T}}(\mathcal{G},y(\mathbb{S}^1)),$$
where $\mathbb{S}^1$ is endowed with its standard topology. If $\mathcal{A}$
is an abelian object of $\mathcal{T}$, then the abelian object $\mathcal{A}^D$
is said to be the \emph{dual} of $\mathcal{A}$.
\end{defn}
For any group object $\mathcal{G}$ of $\mathcal{T}$, there is a canonical
morphism
\begin{equation}\label{mapdansDD}
d_{\mathcal{G}}:\mathcal{G}\longrightarrow \mathcal{G}^{DD}.
\end{equation}
The discussion above shows that if $\mathcal{G}=y(G)$ is represented by a
locally compact abelian topological group $G$, then $y(G)^D$ is
represented by the usual Pontryagin dual
$G^D:=\underline{Hom}_{Top}(G,\mathbb{S}^1)$ of $G$, endowed with the compact-open topology. Therefore, the
following result is given by Pontryagin duality for Hausdorff
locally compact abelian groups.
\begin{thm}\label{Pontryagin}
Let $\mathcal{A}$ be an abelian object of $\mathcal{T}$ representable by an
abelian Hausdorff locally compact topological group. Then one has a
canonical isomorphism
$$d_{\mathcal{A}}:\mathcal{A}\simeq \mathcal{A}^{DD}.$$
\end{thm}
\begin{cor}\label{DD=abelian}
If $y(G)$ is a group of $\mathcal{T}$ represented by a Hausdorff
locally compact topological group $G$, then one has
$$y(G)^{DD}\simeq y(G^{ab}),$$
where $G^{ab}$ is the maximal Hausdorff abelian quotient of $G$.
\end{cor}
\begin{proof}Using Theorem \ref{Pontryagin}, the result follows from
$$y(G)^{D}=y(G^D)=y((G^{ab})^D)=y(G^{ab})^D$$
since $G^{ab}$ is Hausdorff and locally compact.
\end{proof}

\subsection{Fundamental groups}\label{subsect-fund}

Let $\mathcal{S}$ be a topos and let
$t:\mathcal{E}\rightarrow\mathcal{S}$ be a connected and locally
connected topos over $\mathcal{S}$ (i.e. $t^*$ is fully faithful and
has an $\mathcal{S}$-indexed left adjoint, see \cite{elephant}
C3.3). An object $L$ of $\mathcal{E}$ is said to be \emph{locally
constant over $\mathcal{S}$} if there exists a covering morphism
$U\rightarrow e_{\mathcal{E}}$ of the final object of $\mathcal{E}$,
an object $S$ of $\mathcal{S}$ and an isomorphism $L\times U\simeq
f^*S\times U$ over $U$. The object $U$ is then said to split or trivialize
$L$. Let $LC(\mathcal{E})$ be the full subcategory of $\mathcal{E}$
consisting in locally constant objects of $\mathcal{E}$ over
$\mathcal{S}$. We denote by $SLC(\mathcal{E})$ the category of (internal) sums
of locally constant objects (see \cite{Bunge-Moerdijk} Section 2 for an explicit definition). The category $SLC(\mathcal{E})$ is a
topos and one has a canonical \emph{connected} morphism
\begin{equation}\label{map-loc-csts}
\mathcal{E}\rightarrow SLC(\mathcal{E}),
\end{equation}
whose inverse image is the inclusion
$SLC(\mathcal{E})\hookrightarrow\mathcal{E}$. The fact that this
morphism is connected means that its inverse image is fully
faithful, which is obvious here. Note that this morphism is defined
over $\mathcal{S}$.

Assume that the $\mathcal{S}$-topos $\mathcal{E}$ has a point $p$,
i.e. a section $p:\mathcal{S}\rightarrow\mathcal{E}$ of the
structure map $t:\mathcal{E}\rightarrow\mathcal{S}$. Composing $p$
and the morphism (\ref{map-loc-csts}), we obtain a point
$$\widetilde{p}:\mathcal{S}\rightarrow\mathcal{E}\rightarrow
SLC(\mathcal{E})$$ of the topos $SLC(\mathcal{E})$ over
$\mathcal{S}$. The theory of the fundamental group in the context of
topos theory shows the following (see \cite{Moerdijk-Prodiscrete}
and \cite{Bunge-Moerdijk} Section 1). There exists a "pro-discrete
localic group" $\mathcal{G}$ in $\mathcal{S}$ well defined up to a
canonical isomorphism and an equivalence
$$B\mathcal{G}\simeq SLC(\mathcal{E}),$$
where $B\mathcal{G}$ is the classifying topos of $\mathcal{G}$, i.e.
the topos of $\mathcal{G}$-objects in $\mathcal{S}$. Moreover, the
equivalence above identifies the inverse image of the point
$\widetilde{p}:\mathcal{S}\rightarrow SLC(\mathcal{E})$ with the
forgetful functor $B\mathcal{G}\rightarrow\mathcal{S}$.

The topos $\mathcal{E}$ is said to be \emph{locally simply
connected} over $\mathcal{S}$ if there exists one covering morphism
$U\rightarrow e_{\mathcal{E}}$ trivializing all locally constant
objects in $\mathcal{E}$. In this case one has
$SLC(\mathcal{E})=LC(\mathcal{E})$, and the pro-discrete localic
group $\mathcal{G}$ is just a group object of $\mathcal{S}$ (see
\cite{Barr-Diaconescu} or \cite{Bunge-Moerdijk} Section 1). We denote by $\pi_1(\mathcal{E},p)$ this
group object. We get a \emph{connected} morphism
\begin{equation}\label{universal-cover}
\mathcal{E}\longrightarrow LC(\mathcal{E})\simeq
B_{\pi_1(\mathcal{E},p)}
\end{equation}
over $\mathcal{S}$, i.e. a commutative diagram :
\[ \xymatrix{
\mathcal{E}\ar[r]\ar[rd]&B_{\pi_1(\mathcal{E},p)}\ar[d]\\
&\mathcal{S} }\] The morphism (\ref{universal-cover}) into the
classifying topos $B_{\pi_1(\mathcal{E},p)}$ corresponds to a torsor
in $\mathcal{E}$ of group $\pi_1(\mathcal{E},p)$, which is called
the \emph{universal cover} of $\mathcal{E}$ over $\mathcal{S}$.
\begin{defn}
Let $\mathcal{A}$ be an abelian object of $\mathcal{E}$. We define the
\emph{cohomology of $\mathcal{E}$ with value in $\mathcal{S}$} as
follows:
$$H_{\mathcal{S}}^{n}(\mathcal{E},\mathcal{A})=R^n(t_*)\mathcal{A}.$$
\end{defn}
The fundamental group represents the first cohomology group over an arbitrary base topos. More precisely, one has the following result.
\begin{prop}\label{prop-pi1-represents}
Let $\mathcal{E}$ be a connected, locally connected and locally simply connected topos over $\mathcal{S}$ endowed with a point $p$. For any abelian object $\mathcal{A}$ of $\mathcal{S}$,
$t^*\mathcal{A}$ is a constant abelian object of $\mathcal{E}$ over
$\mathcal{S}$ and one has
$$H_{\mathcal{S}}^{1}(\mathcal{E},t^*\mathcal{A})\simeq\underline{Hom}_{\mathcal{S}}(\pi_1(\mathcal{E},p),\mathcal{A})$$
where the right hand side is the internal Hom-group-object in
$\mathcal{S}$ defined as above.
\end{prop}

\subsubsection{Examples}

Let $X$ be a Hausdorff topological space. We denote by
$Sh(X)$ the topos of sheaves of sets on $X$. There exists a
unique map
$$t:Sh(X)\rightarrow\underline{Set}.$$
The topological space $X$ is connected if and only if $t$ is
connected (i.e. if $t^*$ is fully faithful). Let $F$ be a sheaf on
$X$ (i.e. an étalé space $\widetilde{F}\rightarrow X$). If $X$ is
locally connected then $\widetilde{F}$ is locally connected and
$\widetilde{F}$ is the coproduct in $Sh(X)$ of its
connected components. The functor $F\rightarrow\pi_0(\widetilde{F})$
is left adjoint to $t^*$ hence $t$ is a locally connected map of
topoi. Conversely, if $t$ is a locally connected map then $X$ is
locally connected as a topological space. A sheaf $F$ on $X$ is
locally constant if and only if $\widetilde{F}\rightarrow X$ is an
étale cover. Assume that $X$ is locally simply connected and let
$\{U_i\subseteq X,\,i\in I\}$ be an open covering such that $U_i$ is
simply connected. Then any locally constant sheaf on $X$ is
trivialized by $U:=\coprod U_i\rightarrow X$. A point $x\in X$
yields a morphism $p_{x}:\underline{Set}\rightarrow Sh(X)$
(and conversely). The inverse image of this morphism is the stalk
functor $F\rightarrow F_x$. The category $LC(Sh(X))$ is
precisely the category of étale covers of $X$ and the group
$\pi_1(Sh(X),p_x)$ is the usual fundamental group
$\pi_1(X,x)$. In this special case, the equivalence of categories
$$\appl{LC(Sh(X))}{B\pi_1(X,x)}{F}{F_x}$$
is the usual Galois theory for topological spaces. Here
$B\pi_1(X,x)$ is the classifying topos of the discrete group
$\pi_1(X,x)$, i.e. the category of $\pi_1(X,x)$-sets.

Let $\mathcal{S}$ be a topos and let $G$ be a group of
$\mathcal{S}$. We denote by $B_G$ the topos of $G$-objects in
$\mathcal{S}$. The canonical morphism
$$t:B_G\longrightarrow\mathcal{S}$$
is connected, locally connected and locally simply connected.
Indeed, $t$ is connected since $t^*$ is obviously fully faithful.
Moreover, $t$ is locally connected since $t^*$ has a
$\mathcal{S}$-indexed left adjoint given by the quotient functor
$$t_!\mathcal{F}=\mathcal{F}/G:=\underrightarrow{lim}\,( G\times\mathcal{F}\rightrightarrows\mathcal{F}).$$
Note that the inductive limit or coequalizer $\underrightarrow{lim}\,(
G\times\mathcal{F}\rightrightarrows\mathcal{F})$, where the maps are
given by multiplication and projection, always exists in the topos
$\mathcal{S}$. Finally $E_G\rightarrow\{*\}$ trivializes any object,
hence $t$ is locally simply connected. There is a canonical point
$p:\mathcal{S}\rightarrow B_G$, whose inverse image is the forgetful
functor. In this case, the inclusion $LC(B_G)\hookrightarrow B_G$ is
an equivalence (in fact an isomorphism) and the fundamental group
$\pi_1(B_G,p)$ is $G$.

\subsubsection{} Let $t:\mathcal{E}\rightarrow\mathcal{T}$ be a connected and locally connected topos over $\mathcal{T}$ endowed with a $\mathcal{T}$-valued point $p$. The fundamental group $\pi_1(\mathcal{E},p)$ will be called the \emph{topological fundamental group} of $\mathcal{E}$.

\begin{cor}\label{cor-retrouve-pi1-ab}
Let $t:\mathcal{E}\rightarrow\mathcal{T}$ be a connected, locally connected and locally simply connected topos over $\mathcal{T}$ endowed with a $\mathcal{T}$-valued point $p$.
Let $y\mathbb{S}^1$ be the sheaf of $\mathcal{T}$ represented by the standard topological group $\mathbb{S}^1$, and define $\widetilde{\mathbb{S}}^1:=t^*y\mathbb{S}^1$. One has
$$H_{\mathcal{T}}^{1}(\mathcal{E},\widetilde{\mathbb{S}}^1)\simeq\pi_1(\mathcal{E},p)^D.$$
If $\pi_1(\mathcal{E},p)$ is represented by a locally compact group, then $H_{\mathcal{T}}^{1}(\mathcal{E},\widetilde{\mathbb{S}^1})$ is represented by the usual Pontryagin dual $\pi_1(\mathcal{E},p)^D$, and one has
$$H_{\mathcal{T}}^{1}(\mathcal{E},\widetilde{\mathbb{S}}^1)^D\simeq\pi_1(\mathcal{E},p)^{DD}=\pi_1(\mathcal{E},p)^{ab}$$
where $\pi_1(\mathcal{E},p)^{ab}$ is the maximal Hausdorff quotient of $\pi_1(\mathcal{E},p)$.
\end{cor}
\begin{proof}
This follows from Proposition \ref{prop-pi1-represents}, Definition \ref{defn-dual} and Corollary \ref{DD=abelian}.
\end{proof}

\section{Application to the arithmetic fundamental group}

Let $\bar{X}=(Spec\,\mathcal{O}_F,X_{\infty})$ be the Arakelov
compactification of the ring of integers in a number field $F$.
Following the computations of S. Lichtenbaum (see
\cite{Lichtenbaum}), we are looking for a topos $\bar{X}_{L}$ defined
over $\mathcal{T}$ whose cohomology is related the Dedekind zeta
function $\zeta_F(s)$. This conjectural topos $\bar{X}_{L}$ will be called the \emph{conjectural Lichtenbaum topos}.
The topos $\bar{X}_{L}$ should be defined over
$\mathcal{T}$, since the coefficients for this conjectural
cohomology theory should contain the category of locally compact abelian topological
groups. If we denote by $$t:\bar{X}_{L}\longrightarrow\mathcal{T}$$ the structure
map, then we define the sheaf of continuous real valued functions
$\widetilde{\mathbb{R}}$ to be $t^*(y\mathbb{R})$, where
$y\mathbb{R}$ is the abelian object of $\mathcal{T}$ represented by
the standard topological group $\mathbb{R}$.

Following the computations of S. Lichtenbaum, the cohomology of
$\bar{X}_{L}$ must satisfy the following :
\begin{equation}
H^i(\bar{X}_{L},\mathbb{Z})=\mathbb{Z},\,0,\,Pic^1(\bar{X})^D\mbox{
for $i=0,1,2$ respectively, and}
\end{equation}
\begin{equation}
H^i(\bar{X}_{L},\widetilde{\mathbb{R}})=\mathbb{R},\,\mathbb{R},\,0\mbox{
for $i=0,1,2$ respectively.}
\end{equation}

Recall that for any (Grothendieck) topos $\mathcal{E}$, there is a
unique morphism $e:\mathcal{E}\rightarrow\underline{Set}$. The
cohomology of the topos $\mathcal{E}$ with coefficients in
$\mathcal{A}$ is defined by
$$H^n(\mathcal{E},\mathcal{A}):=R^n(e_*)\mathcal{A}.$$
Since the base topos of the topos $\bar{X}_{L}$ is $\mathcal{T}$
instead of $\underline{Set}$, it is natural to consider the cohomology
of $\bar{X}_W$ with value in $\mathcal{T}$. More precisely, the
category $\mathcal{T}$ is thought of as a universe of sets, and we
define
$$H_{\mathcal{T}}^n(\bar{X}_{L},\mathcal{A}):=R^n(t_*)\mathcal{A},$$
for any abelian object $\mathcal{A}$ of $\bar{X}_{L}$. The unique
morphism $\mathcal{T}\rightarrow\underline{Sets}$ is strongly acyclic
(i.e. its direct image is exact) and this point of view is
inoffensive. We should have
\begin{equation}\label{T-coh-Z-coef}
H_{\mathcal{T}}^i(\bar{X}_{L},\mathbb{Z})=\mathbb{Z},\,0,\,Pic^1(\bar{X})^D\mbox{
for $i=0,1,2$ respectively}
\end{equation}
where $\mathbb{Z}$ and $Pic^1(\bar{X})^D$ are the sheaves of
$\mathcal{T}$ represented by the discrete abelian groups
$\mathbb{Z}$ and $Pic^1(\bar{X})^D$. Respectively, the
$\mathcal{T}$-cohomology of $\bar{X}_{L}$ with value in
$\widetilde{\mathbb{R}}$ should be given by
\begin{equation}\label{T-coh-R-coef}
H_{\mathcal{T}}^i(\bar{X}_{L},\widetilde{\mathbb{R}})=y(\mathbb{R}),\,y(\mathbb{R}),\,0\mbox{
for $i=0,1,2$}
\end{equation}
where $y(\mathbb{R})$ is the abelian object of $\mathcal{T}$
represented by the standard topological group $\mathbb{R}$.

\begin{Hypothese}\label{Hyp1}
The topos $\bar{X}_{L}$ is connected, locally connected, locally
simply connected over $\mathcal{T}$, and endowed with a point
$p:\mathcal{T}\rightarrow \bar{X}_{L}$.
\end{Hypothese}
It is natural to expect that $\bar{X}_{L}$ is connected and locally
connected over $\mathcal{T}$. A point $p:\mathcal{T}\rightarrow
\bar{X}_{L}$ should be given by any valuation of the number field $F$.
However, it is not clear that $\bar{X}_{L}$ should be locally simply
connected over $\mathcal{T}$ (for example $\bar{X}_{et}$ is not
locally simply connected over $\underline{Sets}$ in general). But
this assumption can be avoided using the more advanced notion of
localic groups (or pro-groups). We make this assumption to simplify
the following computations.

\begin{Hypothese}\label{Hyp2}
The cohomology of $\bar{X}_{L}$ with value in $\mathcal{T}$ satisfies
\emph{(\ref{T-coh-Z-coef})} and \emph{(\ref{T-coh-R-coef})}.
\end{Hypothese}

\subsection{The abelian arithmetic fundamental group}

\begin{thm}\label{main-thm}
Let $\bar{X}_{L}$ be a topos over $\mathcal{T}$ satisfying Hypothesis
\ref{Hyp1} and \ref{Hyp2}. Then one has an isomorphism of
topological groups:
$$\pi_1(\bar{X}_{L},p)^{DD}\simeq Pic(\bar{X})$$
where $Pic(\bar{X})$ denotes the Arakelov Picard group of the number
field $F$. In particular, if $\pi_1(\bar{X}_{L},p)$ is represented by
a locally compact topological group, then one has an isomorphism of
topological groups:
$$\pi_1(\bar{X}_{L},p)^{ab}\simeq Pic(\bar{X}).$$
\end{thm}
\begin{proof}
By Hypothesis \ref{Hyp1} and section \ref{subsect-fund}, the
fundamental group $\pi_1(\bar{X}_{L},p)$ is well defined as a group
object of $\mathcal{T}$. The basic idea is to use Corollary \ref{cor-retrouve-pi1-ab} to recover the abelian fundamental group. We have
$$H_{\mathcal{T}}^1(\bar{X}_{L},\mathcal{A})=\underline{Hom}_{\mathcal{T}}(\pi_1(\bar{X}_{L},p),\mathcal{A}),$$
for any abelian object $\mathcal{A}$ of $\mathcal{T}$. The exact
sequence of topological groups
$$0\rightarrow\mathbb{Z}\rightarrow\mathbb{R}\rightarrow\mathbb{S}^1\rightarrow0$$
induces an exact sequence
$$0\rightarrow\mathbb{Z}\rightarrow \widetilde{\mathbb{R}}\rightarrow \widetilde{\mathbb{S}}^1\rightarrow0$$
of abelian sheaves in $\bar{X}_{L}$, where $\widetilde{\mathbb{S}}^1$
denotes $t^*(y(\mathbb{S}^1))$. Consider the induced long exact
sequence of $\mathcal{T}$-cohomology
$$0=H_{\mathcal{T}}^1(\bar{X}_{L},\mathbb{Z})\rightarrow H_{\mathcal{T}}^1(\bar{X}_{L},\widetilde{\mathbb{R}})\rightarrow
H_{\mathcal{T}}^1(\bar{X}_{L},\widetilde{\mathbb{S}}^1)\rightarrow
H_{\mathcal{T}}^2(\bar{X}_{L},\mathbb{Z})\rightarrow
H_{\mathcal{T}}^2(\bar{X}_{L},\widetilde{\mathbb{R}})=0$$ We obtain an
exact sequence in $\mathcal{T}$:
$$0\rightarrow\mathbb{R}\rightarrow
H_{\mathcal{T}}^1(\bar{X}_{L},\widetilde{\mathbb{S}}^1)\rightarrow
Pic^1(\bar{X})^D\rightarrow 0$$ It follows that
$$H_{\mathcal{T}}^1(\bar{X}_{L},\widetilde{\mathbb{S}}^1)=\underline{Hom}_{\mathcal{T}}(\pi_1(\bar{X}_{L},p),y(\mathbb{S}^1))=\pi_1(\bar{X}_{L},p)^D$$
is representable by an abelian Hausdorff locally compact topological
group. Indeed,
$H_{\mathcal{T}}^1(\bar{X}_{L},\widetilde{\mathbb{S}}^1)$ is
representable locally on $Pic^1(\bar{X})^D$. But $Pic^1(\bar{X})^D$
is discrete (recall that $Pic^1(\bar{X})$ is compact) and the Yoneda
embedding $y:Top\rightarrow\mathcal{T}$ commutes with coproducts
(see \cite{MatFlach} Cor 1), hence the sheaf
$H_{\mathcal{T}}^1(\bar{X}_{L},\widetilde{\mathbb{S}}^1)$ is
representable by a topological space $T$. The functor
$y:Top\rightarrow\mathcal{T}$ is fully faithful and commutes with
finite projective limits. Hence the space $T$ is endowed with a
structure of an abelian topological group since
$y(T)=H_{\mathcal{T}}^1(\bar{X}_{L},\widetilde{\mathbb{S}}^1)$ is an
abelian object of $\mathcal{T}$. 
The connected component of the identity in $T$ is isomorphic to
$\mathbb{R}$, since $Pic^1(\bar{X})^D$ is discrete. Hence $T$ is
Hausdorff and locally compact. Therefore
$\pi_1(\bar{X}_{L},p)^{DD}=y(T^D)$ is representable by an abelian
Hausdorff locally compact topological group as well.

By Pontryagin duality, we obtain the exact sequence in $\mathcal{T}$
\begin{equation}\label{exact-sequ-fund-group}
0\rightarrow
Pic^1(\bar{X})\rightarrow\pi_1(\bar{X}_{L},p)^{DD}\rightarrow\mathbb{R}\rightarrow0.
\end{equation}
\begin{lem}
One has $H^n(B_{\mathbb{R}},Pic^1(\bar{X}))=0$ for any $n\geq2$.
\end{lem}
\begin{proof}
Let $r_1$ and $r_2$ be the sets of real and complex places of the number field $F$ respectively.
One has the exact sequence of topological groups (with trivial
$\mathbb{R}$-action)
$$0\rightarrow \mathbb{R}^{r_1+r_2-1}/log(\mathcal{O}^{\times}_F/\mu_F)\rightarrow Pic^1(\bar{X})\rightarrow Cl(F)\rightarrow0$$
where $log(\mathcal{O}^{\times}_F/\mu_F)$ denotes the image of the logarithmic
embedding of the units modulo torsion $\mathcal{O}^{\times}_F/\mu_F$ in the the kernel
$\mathbb{R}^{r_1+r_2-1}$ of the sum map
$\Sigma:\mathbb{R}^{r_1+r_2}\rightarrow\mathbb{R}$. The class group
$Cl(F)$ is finite hence we have $H^n(B_{\mathbb{R}},Cl(F))=0$ for
any $n\geq1$ (see \cite{MatFlach} Prop 9.6). Hence we have
$$H^n(B_{\mathbb{R}},\mathbb{R}^{r_1+r_2-1}/log(\mathcal{O}^{\times}_F/\mu_F))\simeq H^n(B_{\mathbb{R}},Pic^1(\bar{X}))$$
for any $n\geq1$. Now consider the exact sequence
$$0\rightarrow\mathcal{O}^{\times}_F/\mu_F\rightarrow\mathbb{R}^{r_1+r_2-1}
\rightarrow\mathbb{R}^{r_1+r_2-1}/log(\mathcal{O}^{\times}_F/\mu_F)\rightarrow0.$$
We have $H^n(B_{\mathbb{R}},\mathcal{O}^{\times}_F/\mu_F)=0$ for any $n\geq1$, since
$\mathcal{O}^{\times}_F/\mu_F$ is discrete (\cite{MatFlach} Prop 9.6). We obtain
$$H^n(B_{\mathbb{R}},Pic^1(\bar{X}))=H^n(B_{\mathbb{R}},\mathbb{R}^{r_1+r_2-1}/log(\mathcal{O}^{\times}_F/\mu_F))
=H^n(B_{\mathbb{R}},\mathbb{R}^{r_1+r_2-1})=0$$
for any $n\geq2$ (again see \cite{MatFlach} Prop 9.6).
\end{proof}

In particular $H^2(B_{\mathbb{R}},Pic^1(\bar{X}))=0$ hence (see
\cite{Giraud} VIII Proposition 8.2) the extension
(\ref{exact-sequ-fund-group}) of abelian groups in $\mathcal{T}$ is
isomorphic to the exact sequence
$$0\rightarrow
Pic^1(\bar{X})\rightarrow
Pic(\bar{X})\rightarrow\mathbb{R}\rightarrow0,
$$
where $Pic(\bar{X})\rightarrow\mathbb{R}$ is the canonical
continuous morphism. In particular there is an isomorphism
$\pi_1(\bar{X}_{L},p)^{DD}\simeq Pic(\bar{X})$ in $\mathcal{T}$. This
shows that $\pi_1(\bar{X}_{L},p)^{DD}$ and $Pic(\bar{X})$ are
isomorphic as topological groups, since $y:
Top\rightarrow\mathcal{T}$ is fully faithful. The last claim of the
theorem then follows from Corollary \ref{DD=abelian}.
\end{proof}

\subsection{The morphism flow and the fundamental class}

\begin{cor}Let $\bar{X}_{L}$ be a topos over $\mathcal{T}$ satisfying Hypothesis
\ref{Hyp1} and \ref{Hyp2}. Then there is a canonical morphism over
$\mathcal{T}$ :
\begin{equation}\label{maptoBpic}
\pi:\bar{X}_{L}\longrightarrow B_{Pic(\bar{X})}.
\end{equation}
In particular, there is a canonical morphism
\begin{equation}\label{maptoBR}\mathfrak{f}:\bar{X}_{L}\longrightarrow B_{\mathbb{R}}.
\end{equation}
\end{cor}
\begin{proof}
There is a morphism $\bar{X}_{L}\rightarrow B_{\pi_1(\bar{X}_{L},p)}$
over $\mathcal{T}$ (the universal cover defined by the point $p$).
Composing with the morphism of classifying topos induced by the map
(see (\ref{mapdansDD}))
$$\pi_1(\bar{X}_{L},p)\longrightarrow\pi_1(\bar{X}_{L},p)^{DD}\simeq Pic(\bar{X}),$$
we get the morphism $\pi$. Note that (\ref{maptoBpic}) does not
depend on $p$ since $\pi_1(\bar{X}_{L},p)^{DD}$ is abelian. The
morphism $\mathfrak{f}$ is then given by the canonical morphism of topological
groups
$$Pic(\bar{X})\longrightarrow\mathbb{R},$$
or (more directly) by the Pontryagin dual of the map
$$\mathbb{R}=H_{\mathcal{T}}^1(\bar{X}_{L},\widetilde{\mathbb{R}})\longrightarrow
H_{\mathcal{T}}^1(\bar{X}_{L},\widetilde{\mathbb{S}}^1)=\pi_1(\bar{X}_{L},p)^{D}.$$
\end{proof}

\begin{cor}
Let $\bar{X}_{L}$ be a topos over $\mathcal{T}$ satisfying Hypothesis
\ref{Hyp1} and \ref{Hyp2}. Then there is a \emph{fundamental class}
$\theta\in H^1(X_{L},\widetilde{\mathbb{R}})$. If the fundamental group $\pi_1(\bar{X}_{L},p)$ is representable by a locally compact group, then
$$\theta\in H^1(X_{L},\widetilde{\mathbb{R}})=Hom_{cont}(Pic(\bar{X}),\mathbb{R})$$
is the canonical continuous morphism $\theta:Pic(\bar{X})\rightarrow\mathbb{R}$.
\end{cor}
\begin{proof} the canonical map $\pi:\bar{X}_{L}\rightarrow B_{Pic(\bar{X})}$
induces a map
$$\pi^*:H^1(B_{Pic(\bar{X})},\widetilde{\mathbb{R}})\longrightarrow H^1(\bar{X}_{L},\widetilde{\mathbb{R}}).$$
The direct image of the unique map
$\mathcal{T}\rightarrow\underline{Sets}$ is exact hence we have
$$H^1(B_{Pic(\bar{X})},\widetilde{\mathbb{R}})=H^0(\mathcal{T},H^1_{\mathcal{T}}(B_{Pic(\bar{X})},\widetilde{\mathbb{R}}))
=Hom_{Top}(Pic(\bar{X}),\mathbb{R}).$$ Therefore the usual
continuous morphism $\alpha:Pic(\bar{X})\rightarrow\mathbb{R}$ is a
distinguished element $\alpha\in
H^1(B_{Pic(\bar{X})},\widetilde{\mathbb{R}})$. We define the
fundamental class as follows :
$$\theta:=\pi^*(\alpha)\in H^1(\bar{X}_{L},\widetilde{\mathbb{R}}).$$
Note that the fundamental class $\varphi$ can also be defined by
$$\theta:=\mathfrak{f}^*(Id_\mathbb{R})\in H^1(\bar{X}_{L},\widetilde{\mathbb{R}}),$$
where $Id_\mathbb{R}$ is the distingushed non-zero element of
$H^1(B_{\mathbb{R}},\widetilde{\mathbb{R}})=Hom_{Top}(\mathbb{R},\mathbb{R})$.

Finally, if the fundamental group $\pi_1(\bar{X}_{L},p)$ is representable by a locally compact group, then the map
$$\pi^*:H^1(B_{Pic(\bar{X})},\widetilde{\mathbb{R}})\longrightarrow H^1(\bar{X}_{L},\widetilde{\mathbb{R}})$$
is an isomorphism, and $\theta$ can be identified with $\alpha$. Indeed, Theorem \ref{main-thm} yields in this case
\begin{eqnarray*}
H^1(\bar{X}_{L},\widetilde{\mathbb{R}})&=&Hom_{cont}(\pi_1(\bar{X}_{L},p),\mathbb{R})\\
&=&Hom_{cont}(\pi_1(\bar{X}_{L},p)^{ab},\mathbb{R})\\
&=&Hom_{cont}(Pic(\bar{X}),\mathbb{R}).
\end{eqnarray*}
\end{proof}

\subsection{The fundamental group and unramified class field theory}\label{subsect-pi1&un-CFT}
There exist complexes $R_W(\varphi_!\mathbb{Z})$ and $R_W(\mathbb{Z})$ of sheaves on the Artin-Verdier \'etale topos whose hypercohomology is the conjectural Lichtenbaum cohomology with and without compact support respectively (see \cite{On the WE}). This suggests the existence of a canonical morphism of topoi
$$\gamma:\bar{X}_{L}\longrightarrow \bar{X}_{et}$$
such that $R\gamma_*\mathbb{Z}=R_W(\mathbb{Z})$, where $\bar{X}_{et}$ denotes the Artin-Verdier étale topos of $X$.
One the one hand, the complex $R_W(\mathbb{Z})$ yields a canonical map
$$H^n(\bar{X}_{et},\mathbb{Z})\longrightarrow H_L^n(\bar{X},\mathbb{Z})$$
for any $n\geq0$. In degree $n=2$, this map
\begin{equation}\label{map-etcoh-to-Wcoh-in-deg2}
Pic(X)^D=(\pi_1(\bar{X}_{et})^{ab})^D=H^2(\bar{X}_{et},\mathbb{Z})\longrightarrow H_L^2(\bar{X},\mathbb{Z}):=Pic^1(\bar{X})^D
\end{equation}
is the dual map of the canonical morphism $Pic^1(\bar{X})\rightarrow Pic(X)=Cl(F)$.
On the other hand, the morphism $\gamma$ would induce a morphism of abelian fundamental
groups
\begin{equation}\label{fund-gr-map}
\pi_1(\bar{X}_{L},p)^{DD}\longrightarrow\pi_1(\bar{X}_{et},q)^{DD}\simeq\pi_1(\bar{X}_{et})^{ab}
\end{equation}
where $q$ is a geometric point of $\bar{X}$ such that the following
diagram commutes :
\[\xymatrix{
\bar{X}_{L}\ar[r]^{\gamma}&\bar{X}_{et}\\
\mathcal{T}\ar[r]^{e_{\mathcal{T}}}\ar[u]_p&\underline{Sets}\ar[u]_q }\] Note that $q$
is uniquely determined by $p$ since the unique map
$e_{\mathcal{T}}: \mathcal{T}\rightarrow\underline{Sets}$ has a canonical section $s$ (see \cite{SGA4} IV.4.10). Indeed, one has $e_{\mathcal{T}}\circ s=Id$ hence
\begin{equation}\label{p-induces-q}
q\simeq q\circ e_{\mathcal{T}}\circ s\simeq \gamma\circ p\circ s.
\end{equation}
The map (\ref{fund-gr-map}) needs to be compatible with the canonical map (\ref{map-etcoh-to-Wcoh-in-deg2}). In other words, the following morphism should be the \emph{reciprocity map of class
field theory} :
\begin{equation}\label{fund-gr-map2}
Pic(\bar{X})\simeq\pi_1(\bar{X}_{L},p)^{DD}\longrightarrow\pi_1(\bar{X}_{et})^{ab}
\end{equation}
More precisely, the diagram
\[ \xymatrix{
Pic(\bar{X})\ar[r]\ar[d]&Pic(X)=
Cl(F)\ar[d]\\
\pi_1(\bar{X}_{L},p)^{DD}\ar[r]^{(\ref{fund-gr-map})}&\pi_1(\bar{X}_{et})^{ab}
}\] should be commutative, where $Pic(\bar{X})\rightarrow
Pic(X)= Cl(F)$ is the canonical map, $Cl(F)\rightarrow
\pi_1(\bar{X}_{et})^{ab}$ is the isomorphism of unramified class
field theory and $Pic(\bar{X})\rightarrow\pi_1(\bar{X}_{L},p)^{DD}$
is the isomorphism defined in Theorem \ref{main-thm}.

\subsection{The fundamental group and the closed embedding $i_v$}\label{subsect-specul-closed-pts}
For any closed point $v$ of $\bar{X}$, i.e. any non-trivial
valuation of the number field $F$, we denote by
$W_{k(v)}:=F_v^{\times}/\mathcal{O}^{\times}_{F_v}$ the Weil group of the residue field $k(v)$ at $v$, where $\mathcal{O}^{\times}_{F_v}$ is the kernel of the valuation $F_v^{\times}\rightarrow\mathbb{R}^{\times}$.

Let ${U}\subseteq{X}$ be an open sub-scheme.
The conjectural Lichtenbaum cohomology with compact support is defined as follows (see the introduction of \cite{Lichtenbaum}) :
$$H_c^*({U},\mathcal{A}):=H^*(\bar{X}_{L},\varphi_!\mathcal{A})$$
where $$\varphi:U_{L}:=\bar{X}_{L}/\gamma^*{U}\longrightarrow\bar{X}_{L}$$
is the canonical open embedding. Consider the exact sequence
\begin{equation}\label{exact-sequ-for-specul}
0\rightarrow\varphi_!\varphi^*\mathcal{A}\rightarrow\mathcal{A}\rightarrow i_*i^*\mathcal{A}\rightarrow0,
\end{equation}
where $i:\textsc{F}:\rightarrow\bar{X}_{L}$ the embedding of the closed complement of the open subtopos $\varphi:U_{L}\rightarrow\bar{X}_{L}$.
The morphism $i$ is a closed embedding so that $i_*$ is exact. We obtain
\begin{equation}\label{cohomlogy-of-specul}
H^n(\textsc{F},i^*\mathcal{A})=H^n(\bar{X}_{L},i_*i^*\mathcal{A}).
\end{equation}
Using (\ref{exact-sequ-for-specul}) and (\ref{cohomlogy-of-specul}), we see that
the conjectural Lichtenbaum cohomology with and without compact support determines the cohomology of the closed sub-topos $\textsc{F}$ (with coefficients in $\mathbb{Z}$ and $\widetilde{\mathbb{R}}$), and we find
$$H^*(\textsc{F},i^*\mathcal{A})=H^*(\textsc{F},\mathcal{A})=H^*(\coprod_{v\in\bar{X}-U}B_{W_{k(v)}},\mathcal{A})$$
for $\mathcal{A}=\mathbb{Z}$ and $\mathcal{A}=\widetilde{\mathbb{R}}$.
This suggests the existence of an equivalence
\begin{equation}\label{specul-closed-subtopos}
\textsc{F}\simeq\coprod_{v\in\bar{X}-U}B_{W_{k(v)}}
\end{equation}
The equivalence (\ref{specul-closed-subtopos}) is indeed satisfied (see \cite{these} Chapter 7) by the Weil-\'etale topos in characteristic $p$ (which is the correct Lichtenbaum topos in this case). Moreover, (\ref{specul-closed-subtopos}) is also predicted by Deninger's program (see \cite{these} Chapter 9). Hence the equivalence (\ref{specul-closed-subtopos}) should hold. Using (\cite{SGA4} IV. Cor. 9.4.3), (\cite{On the WE} Prop. 6.2), and the universal property of sums of topoi, one can prove that (\ref{specul-closed-subtopos}) is equivalent to the existence of a pull-back diagram of topoi :
\[ \xymatrix{
B_{W_{k(v)}}\ar[d]_{i_v}\ar[r]&B^{sm}_{G_{k(v)}}\ar[d]_{u_v}\\
\bar{X}_{L}\ar[r]^{\gamma}&\bar{X}_{et} }\] for any $v$ not in $U$. For an ultrametric place $v$, the morphism
$$u_v:B^{sm}_{G_{k(v)}}\simeq Spec(k(v))_{et}\longrightarrow\bar{X}_{et}$$
is defined by  the scheme map $v\rightarrow\bar{X}$ (see \cite{On the WE} Prop. 6.2) and by a geometric point of $\bar{X}$ over $v$. If $v$ is archimedean, $G_{k(v)}=\{1\}$ and $u_v:\underline{Sets}\rightarrow\bar{X}_{et}$ is the point of the \'etale topos corresponding to $v\in\bar{X}$. In particular, for any closed point $v$ of $\bar{X}$, we have a \emph{closed embedding} of topoi
\begin{equation}\label{specul-iv}
i_v:B_{W_{k(v)}}\longrightarrow \bar{X}_{L}
\end{equation}
where $B_{W_{k(v)}}$ is the classifying topos of $W_{k(v)}$. For any closed point $v$ of $\bar{X}$, the composition
$$B_{W_{k(v)}}\longrightarrow \bar{X}_{L}\longrightarrow B_{Pic(\bar{X})}$$
should be the morphism of classifying topoi $B_{W_{k(v)}}\rightarrow
B_{Pic(\bar{X})}$ induced by the canonical morphism of topological
groups (see (\ref{map-residue-Weil-group-to-Pic}))
$${W_{k(v)}}\longrightarrow{Pic(\bar{X})}.$$

Finally the existence of the morphism (\ref{specul-iv}) is also suggested by the following argument. For an ultrametric place $v$, $B_{W_{k(v)}}$ is the Lichtenbaum topos of $Spec(k(v))$. Hence the existence of the morphism (\ref{specul-iv}) follows from the fact that the map $$\bar{X}\rightsquigarrow\bar{X}_{L},$$
sending a (regular) arithmetic scheme to the topos of sheaves on the Grothendieck site conjectured in \cite{Lichtenbaum},
should be a pseudo-functor from the category of (regular) arithmetic schemes
to the 2-category of topoi.

\section{Expected properties of the conjectural Lichtenbaum topos}\label{sect-expected-properties}

The conjectural Lichtenbaum cohomology is in fact known for any étale
$\bar{X}$-scheme $\bar{U}$, and the arguments of the previous section
give the value of the abelian arithmetic fundamental group of
$\bar{U}$. More precisely, one should have
$$\pi_1(\bar{U}_{L},p_{\bar{U}})^{DD}\simeq C_{\bar{U}}$$
where $C_{\bar{U}}$ is the $S$-idèle class group naturally associated to $\bar{U}$ (see (\ref{equ-defn-CU}) below). Moreover, the study of the complexes $R_W(\mathbb{Z})$ and $R_W(\widetilde{\mathbb{R}})$ defined in \cite{On the WE} yields the
functorial behavior of these isomorphisms. The relation between the arithmetic fundamental group and the \'etale fundamental group is given by the natural maps between \'etale cohomology groups and conjectural Lichtenbaum cohomology groups (see section \ref{subsect-pi1&un-CFT}). Finally, the structure of
the conjectural Lichtenbaum topos at the closed points is dictated
by the conjectural Lichtenbaum cohomology with
compact support (see section \ref{subsect-specul-closed-pts}). Putting those facts together, we obtain a (partial) description of the conjectural Lichtenbaum topos. This description is also suggested by our previous study of the Weil-étale topos in characteristic $p$ (see \cite{these} Chapter 8) and by the work of C. Deninger (see \cite{these} Chapter 9).

\subsection{Notations}
We refer to \cite{On the WE} for the definition of the Artin-Verdier \'etale site of $\bar{X}=\overline{Spec(\mathcal{O}_F)}$. The Artin-Verdier \'etale topos $\bar{X}_{et}$ is the category of sheaves of sets on the Artin-Verdier \'etale site. Let $\bar{U}=(Spec\,\mathcal{O}_{K,S_0},U_{\infty})$ be a connected
étale $\bar{X}$-scheme then we consider the $S$-idèle class group of $K$ endowed with the quotient topology :
\begin{equation}\label{equ-defn-CU}
C_{\bar{U}}:=C_{K,S}=coker(\prod_{w\in\bar{U}}\mathcal{O}^{\times}_{K_w}\rightarrow C_K).
\end{equation}
Here $S$ is the set of places of $K$ not corresponding to a point of
$\bar{U}$, $K_w$ is the completion of $K$ at the place $w$ and $\mathcal{O}^{\times}_{K_w}$ is the kernel of the valuation $K_w^{\times}\rightarrow\mathbb{R}^{\times}$. Note that $C_{\bar{U}}$ is a Hausdorff locally compact group canonically associated to $\bar{U}$.

We define the \emph{Weil group} $W_{k(w)}$ of the "residue field $k(w)$" at
any closed point $w$ of $\bar{U}$ as follows:
$$W_{k(w)}:=K_w^{\times}/\mathcal{O}^{\times}_{K_w}.$$
For any closed point $w\in\bar{U}$, the map $K_w^{\times}\rightarrow C_K$ induces a
continuous morphism
\begin{equation}\label{map-residue-Weil-group-to-CU}
W_{k(w)}:=K_w^{\times}/\mathcal{O}^{\times}_{K_w}\longrightarrow
C_{\bar{U}}.
\end{equation}
Note that we have $W_{k(w)}\simeq\mathbb{Z}$ for $w$ ultrametric and $W_{k(w)}\simeq\mathbb{R}_+^{\times}$ for $w$ archimedean. We denote by $G_{k(w)}:=D_w/I_w$ the Galois group of the residue field $k(w)$, where $D_w$ and $I_w$ are respectively the decomposition and the inertia subgroups of $G_K$ at $w$. Hence $G_{k(w)}$ is the trivial group for $w$ archimedean. There is a canonical morphism
\begin{equation}\label{map-residue-W-G}
W_{k(w)}\longrightarrow G_{k(w)}
\end{equation}
for any closed point $w\in\bar{U}$. We consider the big classifying topos $B_{W_{k(w)}}$ and the small classifying topos $B^{sm}_{G_{k(w)}}$, i.e. the category of continuous $G_{k(w)}$-sets. In particular $B^{sm}_{G_{k(w)}}$ is just the final topos $\underline{Sets}$ for $w$ archimedean. The map (\ref{map-residue-W-G}) induces a morphism of toposes :
$$\alpha_v:B_{W_{k(w)}}\longrightarrow B^{sm}_{G_{k(w)}}.$$

We denote by $\mathcal{T}$ the topos of sheaves on the site $(Top,\mathcal{J}_{op})$, where $Top$ is the category of Hausdorff locally compact spaces endowed with the open cover topology. If one needs to use constant sheaves represented by non-locally compact spaces, then we can define $\mathcal{T}':=(Top^h,\mathcal{J}_{op})$ where $Top^h$ is the category of Hausdorff spaces, and consider the base change
$$\bar{X}_{L}\times_{\mathcal{T}}\mathcal{T}'$$
to obtain a connected and locally connected topos over $\mathcal{T}'$.

Finally, if $\underline{\mathcal{G}}$ is a strict pro-group object of
$\mathcal{T}$ given by a covariant functor $\underline{\mathcal{G}}:I\rightarrow Gr(\mathcal{T})$, where $Gr(\mathcal{T})$ denotes the category of groups in $\mathcal{T}$, and $I$ is a small filtered category. We consider the  pro-abelian group object $\underline{\mathcal{G}}^{DD}$ of
$\mathcal{T}$ defined as the composite functor $$(-)^{DD}\circ\underline{\mathcal{G}}:I\longrightarrow Gr(\mathcal{T})\longrightarrow Ab(\mathcal{T}).$$

Let $t:\mathcal{E}\rightarrow\mathcal{T}$ be a connected and locally connected topos over $\mathcal{T}$, i.e. $t$ is a connected and locally connected morphism. In particular $t^*$ has a left adjoint $t_!$. An object $X$ of $\mathcal{E}$ is said to be \emph{connected over $\mathcal{T}$} if $t_!X$ is the final object of $\mathcal{T}$. A \emph{$\mathcal{T}$-point of $\mathcal{E}$} is a section $s:\mathcal{T}\rightarrow\mathcal{E}$ of the structure map $t$, i.e. $t\circ s$ is isomorphic to $Id_{\mathcal{T}}$.

\subsection{Expected properties}\label{subsect-expected-properties}
\begin{enumerate}

\item \emph{The conjectural Lichtenbaum topos} $\bar{X}_{L}$ \emph{should be naturally associated to $\bar{X}$. There should be a canonical connected morphism from $\bar{X}_{L}$ to the Artin-Verdier étale topos :}
    $$\gamma:\bar{X}_{L}\longrightarrow \bar{X}_{et}.$$\\

\item \emph{The conjectural Lichtenbaum topos} $\bar{X}_{L}$ \emph{should be defined over $\mathcal{T}$. The structure map}
$$t:\bar{X}_{L}\longrightarrow\mathcal{T}$$
\emph{should be connected and locally connected, and} $\bar{X}_{L}$ \emph{should have a $\mathcal{T}$-point $p$. For any connected étale
$\bar{X}$-scheme $\bar{U}$, the object $\gamma^*\bar{U}$ of $\bar{X}_{L}$ should be connected over $\mathcal{T}$.}

It follows that the slice topos
$$\bar{U}_{L}:=\bar{X}_{L}/\gamma^*\bar{U}\longrightarrow\bar{X}_{L}\longrightarrow\mathcal{T}$$ is connected and locally connected over $\mathcal{T}$, for any connected étale
$\bar{X}$-scheme $\bar{U}$, and has a $\mathcal{T}$-point: $$p_{\bar{U}}:\mathcal{T}\longrightarrow \bar{U}_{L}.$$
Then the fundamental group $\pi_1(\bar{U}_{L},p_{\bar{U}})$ is well defined as a prodiscrete localic group in $\mathcal{T}$. Moreover, $\pi_1(\bar{U}_{L},p_{\bar{U}})$ \emph{should be pro-representable by a locally compact strict pro-group}, and we consider this fundamental group as a locally compact pro-group. By Corollary \ref{DD=abelian}, we have
$$\pi_1(\bar{U}_{L},p_{\bar{U}})^{DD}=\pi_1(\bar{U}_{L},p_{\bar{U}})^{ab}=\pi_1(\bar{U}_{L})^{ab}.$$
We have a canonical connected morphism
$$\bar{U}_{L}:=\bar{X}_{L}/\gamma^*\bar{U}\longrightarrow\bar{X}_{et}/\bar{U}=\bar{U}_{et}$$
inducing a morphism
$$\varphi_{\bar{U}}:\pi_1(\bar{U}_{L},p_{\bar{U}})\longrightarrow\pi_1(\bar{U}_{et},q_{\bar{U}})$$
where $q_{\bar{U}}$ is defined by $p_{\bar{U}}$ as in (\ref{p-induces-q}). We obtain a morphism
$$\varphi^{DD}_{\bar{U}}:\pi_1(\bar{U}_{L})^{ab}
=\pi_1(\bar{U}_{L},p_{\bar{U}})^{DD}\longrightarrow\pi_1(\bar{U}_{et},p_{\bar{U}})^{DD}=\pi_1(\bar{U}_{et})^{ab}.$$\\

\item \emph{One should have a canonical isomorphism}
$$r_{\bar{U}}:C_{\bar{U}}\simeq\pi_1(\bar{U}_{L})^{ab}$$
\emph{such that the composition}
$$\varphi_{\bar{U}}^{DD}\circ r_{\bar{U}}:\,\,\,\, C_{\bar{U}}\simeq\pi_1(\bar{U}_{L})^{ab}\longrightarrow\pi_1(\bar{U}_{et})^{ab}$$
\emph{is the reciprocity law of class field theory}. This reciprocity
morphism is defined by the topological class formation
$$(\pi_1(\bar{U}_{et},q_{\bar{U}}),\underrightarrow{lim}\,C_{\bar{V}})$$
where $\bar{V}$ runs over the filtered system of pointed étale cover of $(\bar{U},q_{\bar{U}})$
(see \cite{Neukirch} Proposition 8.3.8 and \cite{Neukirch} Theorem 8.3.12).\\

\item \emph{The isomorphism $r_{\bar{U}}$ should be covariantly
functorial for any map $f:\bar{V}\rightarrow\bar{U}$ of connected \'etale $\bar{X}$-schemes.}
More precisely, such a map induces a morphism of toposes
$$f_{L}:\bar{V}_{L}:=\bar{X}_{L}/\bar{V}\longrightarrow\bar{U}_{L}:=\bar{X}_{L}/\bar{U}$$
hence a morphism of abelian pro-groups in $\mathcal{T}$
$$\widetilde{f}_{L}:\pi_1(\bar{V}_{L})^{ab}\longrightarrow \pi_1(\bar{U}_{L})^{ab}.$$
Then the following
diagram should be commutative
\[ \xymatrix{
\pi_1(\bar{V}_{L})^{ab}\ar[r]^{\,\,\,\,\,\,\,\,\,\,\,\,\,\,r_{\bar{V}}}\ar[d]_{\widetilde{f}_{L}}& C_{\bar{V}}\ar[d]^N\\
\pi_1(\bar{U}_{L})^{ab}\ar[r]^{\,\,\,\,\,\,\,\,\,\,\,\,\,\,r_{\bar{U}}}&C_{\bar{U}}
}\] where $N$ is induced by the norm map.\\

\item \emph{For any Galois étale cover $\bar{V}\rightarrow\bar{U}$ (of étale $\bar{X}$-schemes), the
conjugation action on $\pi_1(\bar{V}_{L})^{ab}$ should
correspond to the Galois action on $C_{\bar{V}}$}. In other words,
the following diagram should be commutative
\[ \xymatrix{
\pi_1(\bar{U}_{L},p_{\bar{U}})\times\pi_1(\bar{V}_{L})^{ab}
\ar[rr]^{\,\,\,\,\,\,\,\,\,\,\,\,\,\,(\varphi_{\bar{U}}, r_{\bar{V}})}\ar[d]& &\pi_1(\bar{U}_{et},q_{\bar{U}})\times C_{\bar{V}}\ar[d]\\
\pi_1(\bar{V}_{L})^{ab}\ar[rr]^{\,\,\,\,\,\,\,\,\,\,\,\,\,\,r_{\bar{V}}}&
&C_{\bar{V}} }\] where the vertical arrows are the conjugation
action of $\pi_1(\bar{U}_{L},p_{\bar{U}})$ on
$\pi_1(\bar{V}_{L})^{ab}$ and the natural action of
$\pi_1(\bar{U}_{et},p_{\bar{U}})$ on $C_{\bar{V}}$.\\

\item \emph{The isomorphism $r_{\bar{U}}$ should be contravariantly functorial for
an étale cover.} More precisely, let $\bar{V}\rightarrow\bar{U}$ be
a \emph{finite} étale map. Then the following diagram should be
commutative
\[ \xymatrix{
\pi_1(\bar{V}_{L})^{ab}\ar[r]^{\,\,\,\,\,\,\,\,\,\,\,\,\,\,r_{\bar{V}}}& C_{\bar{V}}\\
\pi_1(\bar{U}_{L})^{ab}\ar[r]^{\,\,\,\,\,\,\,\,\,\,\,\,\,\,r_{\bar{U}}}\ar[u]_{\mbox{tr}}&C_{\bar{U}}\ar[u]
}\] where the map $C_{\bar{U}}\rightarrow C_{\bar{V}}$ is the
inclusion, and $\mbox{tr}$ is the transfer map defined in Proposition \ref{prop-transfer} below.\\

\item \emph{For any closed point $v$ of $\bar{X}$, one should have pull-back of
topoi :}
\[ \xymatrix{
B_{W_{k(v)}}\ar[d]_{i_v}\ar[r]^{\alpha_v}&B^{sm}_{G_{k(v)}}\ar[d]_{u_v}\\
\bar{X}_{L}\ar[r]^{\gamma}&\bar{X}_{et} }\] Here the morphism
$$u_v:B^{sm}_{G_{k(v)}}\simeq Spec(k(v))_{et}\longrightarrow\bar{X}_{et}$$
is defined by a geometric point of $\bar{X}$ over $v$ and by the
scheme map $v\rightarrow\bar{X}$. The map $\alpha_v$ is induced by
the canonical morphism $W_{k(v)}\rightarrow G_{k(v)}$. It follows
that the morphism $i_v$ is a \emph{closed embedding}.\\

One the one hand, the pull-back
above induces a closed embedding
$$i_w:B_{W_{k(w)}}\longrightarrow\bar{U}_{L}$$ for any $\bar{U}$ étale over
$\bar{X}$ and any closed point $w$ of $\bar{U}$. On the other hand one has a canonical morphism
$$\bar{U}_{L}\rightarrow B_{\pi_1(\bar{U}_{et},p_{\bar{U}})}\rightarrow B_{\pi_1(\bar{U}_{et})^{ab}}\simeq B_{C_{\bar{U}}}.$$
\item \emph{For any closed point $w$ of a connected étale $\bar{X}$-scheme $\bar{U}$, the composition}
$$B_{W_{k(w)}}\longrightarrow \bar{U}_{L}\longrightarrow B_{C_{\bar{U}}}$$
\emph{should be the morphism of classifying topoi induced by the canonical
morphism of topological groups ${W_{k(w)}}\rightarrow{C_{\bar{U}}}$.}\\

Define the sheaf of continuous real valued functions on $\bar{X}_{L}$ as $\widetilde{\mathbb{R}}:=t^*y\mathbb{R}$ where $y\mathbb{R}$ is the sheaf of $\mathcal{T}$ represented by the standard topological group $\mathbb{R}$.
\item \emph{For any \'etale $\bar{X}$-scheme $\bar{U}$, one should have} $H^n(\bar{U}_{L},\widetilde{\mathbb{R}})=0$ \emph{for any $n\geq2$.}

\end{enumerate}

\bigskip

The following result shows that the axioms listed above are consistent. A proof is given in \cite{Fundamental-group-II}.
\begin{thm}
There exists a topos satisfying Axioms $(1)-(9)$ listed above.
\end{thm}
Note that the isomorphism $r_{\bar{U}}:C_{\bar{U}}\simeq\pi_1(\bar{U}_{L})^{ab}$ can be understood in two different ways. On the one hand, one can consider $\pi_1(\bar{U}_{L})^{ab}$ as a usual topological group defined as the projective limit of the topological pro-group $\pi_1(\bar{U}_{L})^{DD}$. Then $r_{\bar{U}}$ is just an isomorphism of topological groups. On the other hand, one can consider $\pi_1(\bar{U}_{L})^{ab}$ and $C_{\bar{U}}$ as topological pro-groups (see section \ref{subsect-S-idele-pro-grp} below), and assume that $r_{\bar{U}}$ is an isomorphism of topological pro-groups. The second point of view is stronger than the first.

\subsection{Explanations}
In this section, we define the morphisms used in the previous description. First of all, the fundamental group $\pi_1(\bar{U}_{L},p_{\bar{U}})$ is assumed to be pro-representable by a locally compact strict pro-group. In other words, we assume that there exist a locally compact strict pro-group $\underline{G}$ indexed over a small filtered category (in the usual sense, see definition \ref{defn-progroup}) and an equivalence $SLC_{\mathcal{T}}(\bar{U}_L)\simeq B_{\underline{G}}$ compatible with the point $p_{\bar{U}}$, where $SLC_{\mathcal{T}}(\bar{U}_L)$ and $B_{\underline{G}}$ are defined as in \cite{Bunge-Moerdijk} Section 2 and as in Definition \ref{defn-classifying-topos-pro-grp} respectively.

The fact that the fundamental groups $\pi_1(\bar{U}_{L},p_{\bar{U}})$ and $\pi_1(\bar{U}_{et},q_{\bar{U}})$ should be defined as topological pro-groups and the previous description of the Lichtenbaum topos suggests that the groups $C_{\bar{U}}$ is in fact a topological pro-group and that all the maps between these topological pro-group are compatible with this additional structure. We show below that it is indeed the case. This detail can be ignored if one considers the limit of those topological pro-groups computed in the category of topological groups, and the morphisms between these pro-groups as usual continuous morphisms.

\subsubsection{The $S$-id\`ele class group as a pro-group}\label{subsect-S-idele-pro-grp}
Let $\bar{U}=(Spec\,\mathcal{O}_{K,S_0},U_{\infty})$ be a connected
étale $\bar{X}$-scheme. We denote by $S_{\infty}$ the set of archimedean places of $K$ not corresponding to a point of $U_{\infty}$, i.e. $U_{\infty}\coprod S_{\infty}$ is the set of archimedean places of $K$. If we set $S=(S_0\cup S_{\infty})$ then we have
$$\bar{U}=\overline{Spec\,\mathcal{O}_{K}}-S$$
and $C_{\bar{U}}=C_{K,S}$ is the $S$-id\`ele class group of $K$. Assume for simplicity that $S_{\infty}\neq\emptyset$. Then there is an exact sequence of topological groups
\begin{equation}\label{exact-sequ-CU-progrp}
0\rightarrow\prod_{w\in S_0}\mathcal{O}^{\times}_{K_w} \rightarrow C_{K,S}\rightarrow C_{K,S_\infty}\rightarrow 0
\end{equation}
where $C_{K,S_{\infty}}$ is the following extension of the finite group $Cl(K)$:
$$0\rightarrow(\prod_{w\in S_{\infty}}K_w^{\times}\prod_{w\in U_{\infty}}\mathbb{R}_+^{\times})/\mathcal{O}^{\times}_{K}\rightarrow C_{K,S_{\infty}}\rightarrow Cl(K)\rightarrow0$$
Note that $C_{K,S_{\infty}}$ has a finite filtration such that the quotients of the form $Fil^n/Fil^{n+1}$ are either finite or connected. Recall that, for $w$ ultrametric, $\mathcal{O}^{\times}_{K_w}$ is given with the filtration
$$\mathcal{O}^{\times}_{K_w}={U}_w^0\supseteq{U}_w^1\supseteq{U}_w^2\supseteq...$$
so that $\mathcal{O}^{\times}_{K_w}$ is the profinite group:
$$\mathcal{O}^{\times}_{K_w}=\underleftarrow{lim}\,{U}_w^0/{U}_w^n.$$
Hence the exact sequence (\ref{exact-sequ-CU-progrp}) provides $C_{\bar{U}}$ with a structure of a topological pro-group. More precisely, one has
$$C_{K,S}=\underleftarrow{lim}\,C_{K,S}/\Omega$$
where $\Omega$ runs over the system of open subgroups of $\prod_{w\in S_0}\mathcal{O}^{\times}_{K_w}$.
\begin{defn}
We define $C_{\bar{U}}$ as the topological pro-group
$$C_{\bar{U}}:=\{C_{K,S}/\Omega,\mbox{ for $\Omega$ open in $\prod_{w\in S_0}\mathcal{O}^{\times}_{K_w}$}\}.$$
\end{defn}
The pro-group $C_{\bar{U}}$ can also be seen as the locally compact group $C_{K,S}$ endowed with the filtration
\begin{equation}\label{filt-onCU}
C_{K,S}\supseteq\prod_{w\in S_0}\mathcal{O}^{\times}_{K_w}\supseteq\prod_{w\in S_0}{U}_w^{1}\supseteq\prod_{w\in S_0}{U}_w^{2}\supseteq...
\end{equation}
Indeed the sequence $\{\Omega^n:=\prod_{w\in S_0}{U}_w^{n},\mbox{ for $n\geq0$}\}$ is cofinal in the system of open $\Omega\subseteq\prod_{w\in S_0}\mathcal{O}^{\times}_{K_w}$. Hence the pro-group $C_{\bar{U}}$ can be defined as follows:
$$C_{\bar{U}}:=\{C_{K,S}/\Omega^n,\mbox{ for $n\geq0$}\}.$$

\begin{prop}
For any map $\bar{V}\rightarrow\bar{U}$ of connected \'etale $\bar{X}$-schemes, the map $N:C_{\bar{V}}\rightarrow C_{\bar{U}}$, induced by the usual norm map, is compatible with the pro-group structures of $C_{\bar{V}}$ and $C_{\bar{U}}$.

For any Galois \'etale cover $\bar{V}\rightarrow\bar{U}$, the usual Galois action of $Gal(\bar{V}/\bar{U})$ on $C_{\bar{V}}$ is compatible with the pro-group structure of $C_{\bar{V}}$.

For any finite \'etale map $\bar{V}\rightarrow\bar{U}$, the natural morphism $C_{\bar{U}}\rightarrow C_{\bar{V}}$ is compatible with the pro-group structures of $C_{\bar{V}}$ and $C_{\bar{U}}$.

For any connected \'etale $\bar{X}$-schemes $\bar{U}$, the reciprocity morphism
$$r_{\bar{U}}:C_{\bar{U}}\longrightarrow \pi_1(\bar{U}_{et})^{ab}$$
is a  morphism of topological pro-groups.
\end{prop}
\begin{proof}
Concerning the first three statements, we just have to remark that those morphisms are all compatible with the filtration (\ref{filt-onCU}), which is clear. The reciprocity morphism $r_{\bar{U}}$ is defined by the topological class formation
$(\pi_1(\bar{U}_{et},q_{\bar{U}}),\underrightarrow{lim}\,C_{\bar{V}})$,
where $\bar{V}$ runs over the filtered system of pointed étale cover of $(\bar{U},q_{\bar{U}})$
(see \cite{Neukirch} Proposition 8.3.8 and \cite{Neukirch} Theorem 8.3.12). Recall that the group $U_v^n$ is mapped, by class field theory, onto the $n^{th}$-ramification subgroup $$(G_v^n)^{ab}\subset G_{K_v}^{ab}\subset G^{ab}_{K,S}=\pi_1(\bar{U}_{et},q_{\bar{U}})^{ab}.$$ Hence $r_{\bar{U}}$ is a morphism of topological pro-groups.
\end{proof}

\subsubsection{The morphism $\varphi_{\bar{U}}$ has dense image.}\label{subsubsect-transfer}
By Axiom (1), the map $\gamma:\bar{X}_{L}\rightarrow\bar{X}_{et}$  is connected, i.e. $\gamma^*$ is fully faithful. It follows immediately that the morphism $$\gamma_{\bar{U}}:\bar{U}_{L}:=\bar{X}_{L}/\gamma^*\bar{U}\longrightarrow\bar{X}_{et}/U=\bar{U}_{et}$$ is connected as well. Chose a $\mathcal{T}$-point $p_{\bar{U}}$ of $\bar{U}_{L}$ and let $q_{\bar{U}}$ be the geometric point of $\bar{U}$ defined by $p_{\bar{U}}$ as in section \ref{subsect-pi1&un-CFT}. We have a commutative square
\[ \xymatrix{
\bar{U}_{L}\ar[r]^{\gamma_{\bar{U}}}\ar[d]& \bar{U}_{et}\ar[d]\\
B_{\pi_1(\bar{U}_{L},p_{\bar{U}})}\ar[r]^{B_{\varphi_{\bar{U}}}}& B^{sm}_{\pi_1(\bar{U}_{et},q_{\bar{U}})}
}\]
where the vertical maps are both connected. Indeed the inverse image of the morphism $\bar{U}_{L}\rightarrow B_{\pi_1(\bar{U}_{L},p_{\bar{U}})}$ (respectively of the morphism $\bar{U}_{et}\rightarrow B^{sm}_{\pi_1(\bar{U}_{et},q_{\bar{U}})}$) is the inclusion of the full subcategory of sums of locally constant objects $SLC_{\mathcal{T}}(\bar{U}_{L})\hookrightarrow\bar{U}_{L}$ (respectively $SLC(\bar{U}_{et})\hookrightarrow\bar{U}_{et}$). Hence the previous diagram shows that $$B_{\varphi_{\bar{U}}}:B_{\pi_1(\bar{U}_{L},p_{\bar{U}})}\longrightarrow B^{sm}_{\pi_1(\bar{U}_{et},q_{\bar{U}})}$$ is connected as well. This morphism is induced by the morphism of strict topological pro-groups:
$$\varphi_{\bar{U}}:\pi_1(\bar{U}_{L},p_{\bar{U}})\longrightarrow\pi_1(\bar{U}_{et},q_{\bar{U}})$$

Consider $\pi_1(\bar{U}_{L},p_{\bar{U}})$ as a projective system of locally compact groups $(W_{\alpha})_{\alpha\in A}$ and $\pi_1(\bar{U}_{et},q_{\bar{U}})$ as a projective system of finite groups $(G_{\beta})_{\beta\in B}$. Then $\varphi_{\bar{U}}$ is given by a family, indexed over $B$, of compatible morphisms $W_{\alpha}\rightarrow G_{\beta}$. More precisely, one has
$$\varphi_{\bar{U}}\in Hom((W_{\alpha})_{\alpha\in A},(G_{\beta})_{\beta\in B}):=\underleftarrow{lim}_{_{\beta\in B}}\, \underrightarrow{lim}_{_{\alpha\in A}}\, Hom_c(W_{\alpha},G_{\beta}).$$

\begin{defn}
We say that $\varphi_{\bar{U}}$ \emph{has dense image} if all those maps $W_{\alpha}\rightarrow G_{\beta}$ are surjective.
\end{defn}
The fact that the morphism $B_{\varphi_{\bar{U}}}$ is connected implies that $\varphi_{\bar{U}}$ has dense image in that sense. Indeed, assume that one of the maps $W_{\alpha}\rightarrow G_{\beta}$ is not surjective. Then the functor $\varphi^*:B^{sm}_{G_{\beta}}\rightarrow B_{W_{\alpha}}$, sending a $G_{\beta}$-set $E$ to the (sheaf represented by) the discrete $W_\alpha$-space $E$ on which $W_\alpha$ acts via $W_{\alpha}\rightarrow G_{\beta}$, is not fully faithful. But we have the commutative diagram of categories
\[ \xymatrix{
B_{\pi_1(\bar{U}_{L},p_{\bar{U}})}& B^{sm}_{\pi_1(\bar{U}_{et},q_{\bar{U}})}\ar[l]_{B^*_{\varphi_{\bar{U}}}}\\
B_{W_{\alpha}}\ar[u]& B^{sm}_{G_{\beta}}\ar[u]\ar[l]_{\varphi^*}
}\]
where the vertical arrows are fully faithful functors. Hence the fact that $\varphi^*$ is not fully faithful implies that $B^*_{\varphi_{\bar{U}}}$ is not fully faithful. We have obtained the following result.

\begin{prop}
Let $\bar{X}_{L}$ be a topos satisfying Axioms $(1)-(9)$. Then for any connected \'etale $\bar{X}$-scheme $\bar{U}$ the morphism of topological pro-groups $\varphi_{\bar{U}}$ has dense image.
\end{prop}

Let $\bar{V}\rightarrow\bar{U}$ be a finite Galois \'etale cover of \'etale $\bar{X}$-schemes with $Gal(\bar{V}/\bar{U})=G$, and consider the injective morphism of topological pro-group
$$\pi_1(\bar{V}_{L},p_{\bar{V}})\hookrightarrow \pi_1(\bar{U}_{L},p_{\bar{U}}).$$
In other words, if one sees the fundamental groups of $\bar{V}_{L}$ and of $\bar{U}_{L}$ as projective systems of topological groups $(W'_{\alpha})_{\alpha\in A}$ and $(W_{\alpha})_{\alpha\in A}$ (indexed over the same filtered category $A$), the previous map is given by a family of compatible injective morphisms of topological groups $W'_{\alpha}\rightarrow W_{\alpha}$ . One can consider the quotient pro-object of $\mathcal{T}$
$$\pi_1(\bar{U}_{L},p_{\bar{U}})/\pi_1(\bar{V}_{L},p_{\bar{V}}):=(yW_{\alpha}/yW'_{\alpha})_{\alpha\in A}.$$
Then this projective system is in fact an essentially constant pro-group and one has an isomorphism in $\mathcal{T}$:
$$\pi_1(\bar{U}_{L},p_{\bar{U}})/\pi_1(\bar{V}_{L},p_{\bar{V}})\simeq y(G).$$
More generally for any finite \'etale map $\bar{V}\rightarrow\bar{U}$ of \'etale $\bar{X}$-schemes the pro-object of $\mathcal{T}$
$$\pi_1(\bar{U}_{L},p_{\bar{U}})/\pi_1(\bar{V}_{L},p_{\bar{V}})$$ is essentially constant, endowed with an action of the pro-group object $\pi_1(\bar{U}_{L},p_{\bar{U}})$, and one has an isomorphism of finite $\pi_1(\bar{U}_{L},p_{\bar{U}})$-sets:
$$\pi_1(\bar{U}_{L},p_{\bar{U}})/\pi_1(\bar{V}_{L},p_{\bar{V}})
\simeq\pi_1(\bar{U}_{et},q_{\bar{U}})/\pi_1(\bar{V}_{et},q_{\bar{V}}).$$
Therefore, for any finite \'etale map $\bar{V}\rightarrow\bar{U}$, the induced morphism $$\pi_1(\bar{V}_{L},p_{\bar{V}})\longrightarrow \pi_1(\bar{U}_{L},p_{\bar{U}})$$ is given by a \emph{compatible family of closed topological subgroups of finite index} $W'_{\alpha}\hookrightarrow W_{\alpha}$. Moreover, one can choose an index $\alpha_0\in A$ such that for any map $\alpha\rightarrow\alpha_0$ in $A$, the map
$W_{\alpha}/W'_{\alpha}\rightarrow W_{\alpha_0}/W'_{\alpha_0}$ is a bijective map of finite sets.
It follows that the usual transfer maps
$$\textrm{tr}_{\alpha}:W^{ab}_{\alpha}\longrightarrow W'^{ab}_{\alpha}$$
are well defined and that they make commutative the following square:
\[ \xymatrix{
W'^{ab}_{\alpha}\ar[d]& W^{ab}_{\alpha}\ar[d]\ar[l]_{tr_{\alpha}}\\
W'^{ab}_{\alpha_0}& W^{ab}_{\alpha_0}\ar[l]_{tr_{\alpha_0}}
}\]
We obtain a morphism of locally compact topological pro-groups
$$\textrm{tr}:\pi_1(\bar{U}_{L},p_{\bar{U}})^{ab}\longrightarrow \pi_1(\bar{V}_{L},p_{\bar{V}})^{ab}$$
If $\bar{V}\rightarrow\bar{U}$ is a Galois \'etale cover, then $W'_{\alpha}$ is normal in $W_{\alpha}$ for any $\alpha\in A$, hence
$W_{\alpha}$ acts on $W'^{ab}_{\alpha}$ by conjugation. This action is certainly functorial in $\alpha$ hence $\pi_1(\bar{U}_{L},p_{\bar{U}})$ acts on $\pi_1(\bar{V}_{L},p_{\bar{V}})^{ab}$ by conjugation. More precisely, we consider the topological pro-group
$$
\fonc{\pi_1(\bar{U}_{L},p_{\bar{U}})\times\pi_1(\bar{V}_{L},p_{\bar{V}})^{ab}}{A}{Gr(\mathcal{T})}{\alpha}
{W_{\alpha}\times W'^{ab}_{\alpha}}$$
Then we have a morphism of topological pro-groups :
$$\pi_1(\bar{U}_{L},p_{\bar{U}})\times\pi_1(\bar{V}_{L},p_{\bar{V}})^{ab}
\longrightarrow\pi_1(\bar{V}_{L},p_{\bar{V}})^{ab}.$$
We have shown the following result.
\begin{prop}\label{prop-transfer}
Let $\bar{V}\rightarrow\bar{U}$ be a finite \'etale map of \'etale $\bar{X}$-schemes. We have a morphism of abelian topological pro-groups
$$\emph{\textrm{tr}}:\pi_1(\bar{U}_{L},p_{\bar{U}})^{ab}\longrightarrow \pi_1(\bar{V}_{L},p_{\bar{V}})^{ab}.$$
If $\bar{V}\rightarrow\bar{U}$ is Galois, then $\pi_1(\bar{U}_{L},p_{\bar{U}})$ acts on $\pi_1(\bar{V}_{L},p_{\bar{V}})^{ab}$ by conjugation: $$\pi_1(\bar{U}_{L},p_{\bar{U}})\times\pi_1(\bar{V}_{L},p_{\bar{V}})^{ab}
\longrightarrow\pi_1(\bar{V}_{L},p_{\bar{V}})^{ab}.$$

\end{prop}

\subsection{The Weil-étale topos in charateristic $p$} Let $Y$
be a smooth projective curve over a finite field $k$. Assume that $Y$ is geometrically connected.
The Weil-étale topos $Y_W$ is defined as the category of $W_k$-equivariant étale sheaves on the geometric curve
$Y\times_k\bar{k}$. Then one can prove that $Y_W$ satisfies all the
properties (1)-(9) above (replacing $\mathcal{T}$ with
$\underline{Sets}$). Moreover $Y_W$ is universal for these
properties, i.e. it is the smallest topos satisfying those
properties. In other words, if $\mathcal{S}$ is a topos satisfying
those properties (1)-(9), then there exists a essentially unique
morphism $\mathcal{S}\rightarrow Y_W$ compatible with this structure
(i.e. making all the diagrams commutative).

We give below a sketch of the proof of these facts. By
(\cite{these} Theorem 8.5) we have a canonical equivalence
\begin{equation}\label{fiber-product-in-Fund}
Y_W\simeq Y_{et}\times_{B_{G_k}}B_{W_k}
\end{equation}
where $Y_{et}$ denotes the étale topos of $Y$, i.e. the category of
sheaves of sets on the étale site of $Y$. Consider the first
projection
$$\gamma:Y_W\simeq Y_{et}\times_{B_{G_k}}B_{W_k}\longrightarrow
Y_{et}.$$ For any étale $Y$-scheme $U$, we thus have
\begin{equation}\label{local-char-p}
Y_W/\gamma^*yU\simeq
(Y_{et}/yU)\times_{B_{G_k}}B_{W_k}\simeq U_{et}\times_{B_{G_k}}B_{W_k}\simeq U_W.
\end{equation}
If $U$ is connected étale over $Y$, then $U_{et}$ and $U_{W}$ are both connected and
locally connected over $\underline{Sets}$. Any geometric point $p_U$
of $U$ yields a $\underline{Sets}$-valued point of $U_{et}$ and
of $U_W$, and we have an isomorphism of pro-discrete groups
\begin{equation}\label{fund-grp-char-p}
\pi_1(U_W,p)\simeq \pi_1(U_{et},p)\times_{G_k}W_k.
\end{equation}
The group $\pi_1(U_{et},p)\times_{G_k}W_k$ is often called the Weil group
of $U$. For any closed point $y$ of $Y$ we have a closed embedding
$$B_{G_{k(y)}}\simeq(Spec\,k(y))_{et}\longrightarrow Y_{et}.$$
The inverse image of this closed
subtopos under $\gamma$ is given by the fiber product:
\begin{equation}\label{closedpoint-charp}
Y_W\times_{Y_{et}}B_{G_{k(y)}}\simeq
B_{W_{k}}\times_{B_{G_{k}}}Y_{et}\times_{Y_{et}}B_{G_{k(y)}}\simeq
B_{W_{k}}\times_{B_{G_{k}}}B_{G_{k(y)}}\simeq B_{W_{k(y)}}.
\end{equation}
Then properties (1)-(9) and the fact that $Y_W$ is universal for
those properties follow from (\ref{fiber-product-in-Fund}),
(\ref{local-char-p}), (\ref{fund-grp-char-p}),
(\ref{closedpoint-charp}) and class field theory for function
fields. Note that Weil's interpretation of class field theory for function fields can be restated as follows:
the reciprocity morphism gives an isomorphism between the S-idele class group
and the abelian Weil-\'etale fundamental group. Finally, note that the canonical morphism
$$Y_W\longrightarrow B_{W_k}$$ gives rise to the Frobenius-equivariant
l-adic cohomology (see \cite{these} Chapter 8 section 4.3).

\subsection{The Lichtenbaum topos and Deninger's dynamical system}\label{subsect-Deninger's-DS}

Axiom (3) of section \ref{subsect-expected-properties} yields a canonical morphism flow
$$\mathfrak{f}:\bar{X}_{L}\longrightarrow B_{Pic(\bar{X})}\longrightarrow B_{\mathbb{R}}.$$
A topos is Grothendieck's generalization of a space hence $\bar{X}_{L}$ can be seen as a generalized space. Then the morphism of topoi $\mathfrak{f}$ can be interpreted as follows. The topos $\bar{X}_{L}$ is a generalized space endowed with an action of the topological group $\mathbb{R}$.

Axioms (7) and (8) above give a closed embedding
$i_{v}:B_{W_{k(v)}}\rightarrow\bar{X}_{L}$ such that the composition
$$i_{v}:B_{W_{k(v)}}\longrightarrow\bar{X}_{L}\longrightarrow B_{\mathbb{R}}$$
is the morphism induced by the canonical morphism
$l_v:W_{k(v)}\rightarrow\mathbb{R}$. For an ultrametric place $v$ of
$F$, the morphism $l_v$ sends the canonical generator of $W_{k(v)}$
to $-log\,N(v)$. The closed embedding $i_v$ should be thought of as
a \emph{closed orbit of the flow of length $log\,N(v)$}. For an archimedean place $v$, the composite morphism $\mathfrak{f}\circ i_v:B_{W_{k(v)}}\rightarrow
B_{\mathbb{R}}$ is an isomorphism of topoi, and $i_v$ should be thought of as
a \emph{(closed) inclusion of a fixed point of the flow}. Thus the
morphism $\mathfrak{f}$ encodes all the numbers $log\,N(v)$. This suggests that the morphism $\mathfrak{f}$, or more
precisely the derived functor $R\mathfrak{f}_*$, could yield a
geometric cohomology theory (i.e. a cohomology allowing a
cohomological interpretation of the zeta function itself). In other
words, one can dream that the conjectural Lichtenbaum topos of $\bar{X}$ (if
it exists), will play the role of Deninger's dynamical system (see \cite{Deninger} for example). In any case, the correct conjectural Lichtenbaum topos is far from being constructed. We refer to (\cite{these} Chapter 9) for some details concerning the analogy between the conjectural Lichtenbaum topos and Deninger's dynamical system.

\subsection{The base topos $B_\mathbb{R}$ and the field with one element $\mathbb{F}_1$}

Let $Y$ a smooth scheme of finite type over $\mathbb{F}_q$. Assume for simplicity that $Y$ is geometrically connected. The Weil-étale topos $Y_W$ is given with a canonical morphism
$$f_Y:Y_W\longrightarrow B^{sm}_{W_{\mathbb{F}_q}}$$
over the small classifying topos $B^{sm}_{W_{\mathbb{F}_q}}$. The Weil-étale topos of $Y$ is thought of as a space endowed with an action of the group  $W_{\mathbb{F}_q}$. Indeed, $Y_W$ is the \'etale topos associated to the Frobenius-equivariant geometric scheme  $\overline{Y}:=Y\otimes_{\mathbb{F}_q}\overline{\mathbb{F}_q}$.
The \'etale topos $\overline{Y}_{et}$ of the geometric scheme  $\overline{Y}$ is obtained as the localization
$$\overline{Y}_{et}=Y_W\times_{B^{sm}_{W_{\mathbb{F}_q}}}\underline{Set}=Y_W/f_Y^*EW_{\mathbb{F}_q}$$
where $\underline{Set}\rightarrow B^{sm}_{W_{\mathbb{F}_q}}$ is the canonical point of $B^{sm}_{W_{\mathbb{F}_q}}$
We denote the \emph{geometric topos} by
$$Y_{\mathrm{geo}}:=Y_W/f_Y^*EW_{\mathbb{F}_q}.$$
We have an exact sequence of fundamental groups
\begin{equation}\label{exact-sequence-fundgrps}
1\rightarrow\pi_1(Y_{\mathrm{geo}},p)\rightarrow\pi_1(Y_W,p)\rightarrow W_{\mathbb{F}_q}\rightarrow1.
\end{equation}

Over $Spec(\mathbb{Z})$, the role of $B_{W_{\mathbb{F}_q}}$ is played by $B_{\mathbb{R}^{\times}_{+}}=B_{\mathbb{R}}$. Let $\bar{X}$ be the compactification of $Spec(\mathcal{O}_F)$. One has a canonical morphism
$$\mathfrak{f}:\bar{X}_{L}\longrightarrow B_{Pic(\bar{X})}\longrightarrow B_{\mathbb{R}^{\times}_{+}}.$$
One can imagine that the base topos $B_{\mathbb{R}^{\times}_{+}}$ is the classifying topos of the Weil group $W_{\mathbb{F}_1}=\mathbb{R}^{\times}_{+}$ of some arithmetic object $\mathbb{F}_1$. Then the localized topos
$$\bar{X}_{\mathrm{geo}}:=\bar{X}_{L}\times_{B_{W_{\mathbb{F}_1}}}\mathcal{T}=
\bar{X}_{L}/\mathfrak{f}^*EW_{\mathbb{F}_1}$$
where $\mathcal{T}\rightarrow B_{\mathbb{R}}$ is the canonical point of $B_{\mathbb{R}}$, would play the role of the geometric \'etale topos $Y_{\mathrm{geo}}:=\overline{Y}_{et}$. Intuitively, $\bar{X}_{\mathrm{geo}}$ corresponds to Deninger's space without the $\mathbb{R}$-action. We have an exact sequence of fundamental groups
$$1\rightarrow\pi_1(\bar{X}_{\mathrm{geo}},p)\rightarrow\pi_1(\bar{X}_{L},p)\rightarrow W_{\mathbb{F}_1}\rightarrow1.$$
This exact sequence is analogous to (\ref{exact-sequence-fundgrps}).

\section{Cohomology}\label{Sect-Cohomology}

In this section we consider the curve $\bar{X}=\overline{Spec(\mathcal{O}_F)}$ where the number field $F$ is totally imaginary. Let $\gamma:\bar{X}_{L}\rightarrow\bar{X}_{et}$ be any topos satisfying Axioms (1)-(9) of the description given in section \ref{subsect-expected-properties}. We show that these axioms yield a natural proof of the fact that the complex of \'etale sheaves $\tau_{\leq2}R\gamma_*(\varphi_!\mathbb{Z})$ produces the special value of $\zeta_F(s)$ at $s=0$ up to sign.

\subsection{The base change from the Weil-étale cohomology to the étale cohomology}

Recall that we denote by $C_{\bar{U}}=C_{K,S}$ the $S$-idele class group canonically associated to $\bar{U}$.
We consider the sheaves on $\bar{U}_{L}$ defined by  $\widetilde{\mathbb{R}}:=t^*_{\bar{U}}(y\mathbb{R})$ and $\widetilde{\mathbb{S}}^1:=t^*_{\bar{U}}(y\mathbb{S}^1)$, where $y\mathbb{S}^1$ and $y\mathbb{R}$ are the sheaves on $\mathcal{T}$ represented by the topological groups $\mathbb{S}^1$ and $\mathbb{R}$, and $t_{\bar{U}}:\bar{U}_{L}\rightarrow\mathcal{T}$ is the canonical map (defined for the second Axiom).

\begin{prop}
For any connected \'etale $\bar{X}$-scheme $\bar{U}$, we have
\begin{eqnarray*}
H^n(\bar{U}_{L},\widetilde{\mathbb{R}})
&=&\mathbb{R}\mbox{ for $n=0$}\\
&=&Hom_{c}(C_{\bar{U}},\mathbb{R}) \mbox{ for $n=1$}\\
&=&0 \mbox{ for $n\geq2$}.
\end{eqnarray*}
\end{prop}
\begin{proof}
The result for $n=0$ follows from the connectedness of $\bar{U}_{L}\rightarrow\mathcal{T}$ given by Axiom (2). Indeed, one has
$$H^0(\bar{U}_W,t^*\widetilde{\mathbb{R}}):=(e_{\mathcal{T}*}\circ t_*)\,t^*\widetilde{\mathbb{R}}=e_{\mathcal{T}*}\widetilde{\mathbb{R}}=\mathbb{R}.$$
where $e_{\mathcal{T}}$ denotes the unique map
$e_{\mathcal{T}}:\mathcal{T}\rightarrow\underline{Sets}$. By Axiom (3), the result for $n=1$ follows from
\begin{eqnarray*}
H^1(\bar{U}_{L},\widetilde{\mathbb{R}})&=&Hom_{c}(\pi_1(\bar{U}_{L}),\mathbb{R})\\
&:=&\underrightarrow{lim}\,Hom_{c}(\pi_1(\bar{U}_{L}),\mathbb{R})\\
&=&\underrightarrow{lim}\,Hom_{c}(\pi_1(\bar{U}_{L})^{ab},\mathbb{R})\\
&=&Hom_{c}(C_{\bar{U}},\mathbb{R}).
\end{eqnarray*}
The result for $n\geq2$ is given by Axiom (9).
\end{proof}

The maximal compact subgroup of $C_{\bar{U}}$, i.e. the kernel of the absolute value map $C_{\bar{U}}\rightarrow\mathbb{R}_{>0}$, is denoted by $C^1_{\bar{U}}$. Thus we have an exact sequence of topological groups
$$1\rightarrow C^1_{\bar{U}}\rightarrow C_{\bar{U}}\rightarrow\mathbb{R}_{>0}\rightarrow 1.$$
The Pontraygin dual $(C^1_{\bar{U}})^D$ is discrete since $C^1_{\bar{U}}$ is compact.
\begin{prop}
For any connected \'etale $\bar{X}$-scheme $\bar{U}$, we have canonically
\begin{eqnarray*}
H^n(\bar{U}_{L},\mathbb{Z})
&=&\mathbb{Z}\mbox{ for $n=0$}\\
&=&0 \mbox{ for $n=1$}\\
&=&(C^1_{\bar{U}})^D \mbox{ for $n=2$}.
\end{eqnarray*}
\end{prop}
\begin{proof}
As above, the result for $n=0$ follows from the connectedness of $\bar{U}_{L}\rightarrow\mathcal{T}$ given by Axiom (2). By Axiom (3), the result for $n=1$ follows from $$H^1(\bar{U}_{L},\mathbb{Z})=Hom_{c}(\pi_1(\bar{U}_{L})^{ab},\mathbb{Z})=0.$$
By Axiom (3) we have canonical isomorphisms
\begin{eqnarray*}
H^1(\bar{U}_{L},\widetilde{\mathbb{S}}^1)&=&Hom_{c}(\pi_1(\bar{U}_{L})^{ab},\mathbb{S}^1)\\
&:=&\underrightarrow{lim}\,Hom_{c}(\pi_1(\bar{U}_{L})^{ab},\mathbb{S}^1)\\
&=&Hom_{c}(\underleftarrow{lim}\,\pi_1(\bar{U}_{L})^{ab},\mathbb{S}^1)\\
&=&Hom_{c}(C_{\bar{U}},\mathbb{S}^1)=C_{\bar{U}}^D.
\end{eqnarray*}
The exact sequence of topological groups
$$0\rightarrow\mathbb{Z}\rightarrow\mathbb{R}\rightarrow\mathbb{S}^1\rightarrow0$$
induces an exact sequence
$$0\rightarrow\mathbb{Z}\rightarrow \widetilde{\mathbb{R}}\rightarrow \widetilde{\mathbb{S}}^1\rightarrow0$$
of abelian sheaves on $\bar{U}_{L}$. The induced long exact sequence
$$0=H^1(\bar{U}_{L},\mathbb{Z})\rightarrow H^1(\bar{U}_{L},\widetilde{\mathbb{R}})\rightarrow
H^1(\bar{U}_{L},\widetilde{\mathbb{S}}^1)\rightarrow
H^2(\bar{U}_{L},\mathbb{Z})\rightarrow
H^2(\bar{U}_{L},\widetilde{\mathbb{R}})=0$$
is canonically identified with
$$0\rightarrow Hom_{c}(C_{\bar{U}},\mathbb{R})\rightarrow Hom_{c}(C_{\bar{U}},\mathbb{S}^1)\rightarrow
H^2(\bar{U}_{L},\mathbb{Z})\rightarrow 0$$
and we obtain $H^2(\bar{U}_{L},\mathbb{Z})=(C^1_{\bar{U}})^D$.
\end{proof}

Recall that we have a canonical morphism $\gamma:\bar{X}_{L}\rightarrow\bar{X}_{et}$.
We consider the truncated functor $\tau_{\leq 2}R\gamma_*$ of the total derived functor $R\gamma_*$.
\begin{cor}
We have $\gamma_*\mathbb{Z}=\mathbb{Z}$, $R^1(\gamma_*)\mathbb{Z}=0$ and $R^2(\gamma_*)\mathbb{Z}$ is the \'etale sheaf associated to the abelian presheaf
$$\fonc{\mathcal{P}^2\gamma_*\mathbb{Z}}{Et_{\bar{X}}}{\underline{Ab}}{\bar{U}}
{(C^1_{\bar{U}})^{D}}.$$
\end{cor}
\begin{proof}
The sheaf $R^n(\gamma_*)\mathbb{Z}$ is the sheaf associated to the presheaf $\bar{U}\mapsto H^{n}(\bar{X}_{L}/\gamma^*\bar{U},\mathbb{Z})$. Hence the Corollary follows immediately from the previous Proposition. Note that it follows from Axiom (4) that the restriction map
$$\mathcal{P}^2\gamma_*\mathbb{Z}(\bar{U})=(C^1_{\bar{U}})^D\rightarrow \mathcal{P}^2\gamma_*\mathbb{Z}(\bar{V})=(C^1_{\bar{V}})^D$$
is the Pontryagin dual of the canonical morphism $C^1_{\bar{V}}\rightarrow C^1_{\bar{U}}$ (induced by the norm map), for any $\bar{V}\rightarrow\bar{U}$ in $Et_{\bar{X}}$.
\end{proof}

The cohomology of sheaf $R^2\gamma_*\mathbb{Z}$ associated to $\mathcal{P}^2\gamma_*\mathbb{Z}$ is computed in Section \ref{sect-R2Z}. The étale sheaf $R^2\gamma_*\mathbb{Z}$ is acyclic for the global sections
functor on ${\bar{X}}_{et}$. More precisely, we have
\begin{eqnarray*}
H^n(\bar{X}_{et};R^2\gamma_*\mathbb{Z})&=& Hom(\mathcal{O}^{\times}_F,\mathbb{Q}) \mbox{ for }n=0,\\
                           &=&0 \mbox{ for }n\geq1.
\end{eqnarray*}

Recall that $Pic(\bar{X})=C_{\bar{X}}$ is the Arakelov Picard group of $F$, and that $\mu_F$ is the group of roots of unity in $F$. We compute below the hypercohomology of the complex of abelian \'etale sheaves $\tau_{\leq 2}R\gamma_*\mathbb{Z}$.
\begin{thm}\label{thm-coho-tau-RgZ}
One has
\begin{eqnarray*}
\mathbb{H}^n(\bar{X}_{et},\tau_{\leq 2}R\gamma_*\mathbb{Z})
&=&\mathbb{Z}\mbox{ for $n=0$}\\
&=&0 \mbox{ for $n=1$}\\
&=&Pic^1(\bar{X})^D \mbox{ for $n=2$}\\
&=&\mu_F^D \mbox{ for $n=3$}\\
&=&0 \mbox{ for $n\geq4$}\\
\end{eqnarray*}
\end{thm}
Recall that the Artin-Verdier \'etale cohomology of $\mathbb{Z}$ is given by
$$H^i(\bar{X}_{et},\mathbb{Z})=\mathbb{Z},\,0,\,Cl(F)^D,\,Hom(\mathcal{O}_F^*,\,\mathbb{Q}/\mathbb{Z}),\,0\mbox{ for $i=0,1,2,3$ and $i\geq4$ respectively.}$$
\begin{proof}
The hypercohomology spectral sequence
$$H^i(\bar{X}_{et},H^j(\tau_{\leq 2}R\gamma_*\mathbb{Z}))\Rightarrow\mathbb{H}^{i+j}(\bar{X}_{et},\tau_{\leq 2}R\gamma_*\mathbb{Z})$$
gives first $\mathbb{H}^0(\bar{X}_{0},\tau_{\leq 2}R\gamma_*\mathbb{Z})=\mathbb{Z}$ and $\mathbb{H}^1(\bar{X}_{et},\tau_{\leq 2}R\gamma_*\mathbb{Z})=0$. One the other hand we have $$\mathbb{H}^n(\bar{X}_{et},\tau_{\leq 2}R\gamma_*\mathbb{Z})=\mathbb{H}^n(\bar{X}_{et},R\gamma_*\mathbb{Z})={H}^n(\bar{X}_{L},\mathbb{Z})$$
for any $n\leq2$. In particular, we have
$$\mathbb{H}^2(\bar{X}_{et},\tau_{\leq 2}R\gamma_*\mathbb{Z})=H^2(\bar{X}_{L},\mathbb{Z})=Pic^1(\bar{X})^D.$$
Therefore, the spectral sequence above yields an exact sequence
$$0\rightarrow Cl(F)^D\rightarrow Pic^1(\bar{X})^D \rightarrow Hom(\mathcal{O}_F^*,\,\mathbb{Q})\rightarrow Hom(\mathcal{O}_F^*,\,\mathbb{Q}/\mathbb{Z})\rightarrow \mathbb{H}^3(\bar{X}_{L},\tau_{\leq 2}R\gamma_*\mathbb{Z})\rightarrow 0.$$
Here the maps $Cl(F)^D\rightarrow Pic^1(\bar{X})^D \rightarrow Hom(\mathcal{O}_F^*,\,\mathbb{Q})$ are explicitly given. The first is induced by the morphism from the Weil-\'etale fundamental group to the \'etale fundamental group
$$\pi_1(\bar{X}_{L})^{ab}\longrightarrow \pi_1(\bar{X}_{et})^{ab}$$
given by Axiom (3), and the second is induced by the canonical morphism $\mathcal{P}^2\gamma_*\mathbb{Z}\rightarrow R^2(\gamma_*)\mathbb{Z}$ (the map from a presheaf to its associated sheaf). It follows that the cokernel of the map
$$Pic^1(\bar{X})^D \rightarrow Hom(\mathcal{O}_F^*,\,\mathbb{Q})$$ is
$$Hom(\mathcal{O}_F^*,\,\mathbb{Q})/Hom(\mathcal{O}_F^*,\,\mathbb{Z})\simeq Hom(\mathcal{O}_F^*/\mu_F,\,\mathbb{Q}/\mathbb{Z}).$$
We obtain an exact sequence
\begin{equation}\label{last-exact-sequence}
0\rightarrow Hom(\mathcal{O}_F^*/\mu_F,\,\mathbb{Q}/\mathbb{Z})\rightarrow Hom(\mathcal{O}_F^*,\,\mathbb{Q}/\mathbb{Z})
\rightarrow \mathbb{H}^3(\bar{X}_{et},\tau_{\leq 2}R\gamma_*\mathbb{Z})\rightarrow 0.
\end{equation}
Let us denote by $\alpha:Hom(\mathcal{O}_F^*/\mu_F,\,\mathbb{Q}/\mathbb{Z})\rightarrow Hom(\mathcal{O}_F^*,\,\mathbb{Q}/\mathbb{Z})$ the first map of the exact sequence (\ref{last-exact-sequence}). We need to show that $\alpha$ is the canonical map. There is a decomposition
$$Hom(\mathcal{O}_F^*,\,\mathbb{Q}/\mathbb{Z})=Hom(\mathcal{O}_F^*/\mu_F,\,\mathbb{Q}/\mathbb{Z})\times \mu_F^D$$
and the composition (where $p$ is the projection)
$$p\circ \alpha:Hom(\mathcal{O}_F^*/\mu_F,\,\mathbb{Q}/\mathbb{Z})\rightarrow Hom(\mathcal{O}_F^*,\,\mathbb{Q}/\mathbb{Z})
\rightarrow \mu_F^D$$
must be 0 since $Hom(\mathcal{O}_F^*/\mu_F,\,\mathbb{Q}/\mathbb{Z})$ is divisible and $\mu_F^D$ finite. It follows that the image of $\alpha$
is contained in the subgroup
$$Hom(\mathcal{O}_F^*/\mu_F,\,\mathbb{Q}/\mathbb{Z})\subset Hom(\mathcal{O}_F^*,\,\mathbb{Q}/\mathbb{Z})$$
hence $\alpha$ induces an injective morphism
$$\widetilde{\alpha}:Hom(\mathcal{O}_F^*/\mu_F,\,\mathbb{Q}/\mathbb{Z})\hookrightarrow Hom(\mathcal{O}_F^*/\mu_F,\,\mathbb{Q}/\mathbb{Z}).$$
Since those two groups are both finite sums of  $\mathbb{Q}/\mathbb{Z}$'s, this map $\widetilde{\alpha}$ needs to be an isomorphism. Indeed, the $n$-torsion subgroups are both finite of same cardinality for any $n$, hence an injective map must be bijective.
Hence $\alpha$ has the same image as the canonical map, i.e. the image of the map induced by the quotient map $\mathcal{O}_F^*\rightarrow \mathcal{O}_F^*/\mu_F$. Hence the exact sequence (\ref{last-exact-sequence}) yields a canonical identification
$$\mathbb{H}^3(\bar{X}_{et},\tau_{\leq 2}R\gamma_*\mathbb{Z})=\mu_F^D.$$

Finally, one has
$$\mathbb{H}^n(\bar{X}_{et},\tau_{\leq 2}R\gamma_*\mathbb{Z})=0\mbox{ for $n\geq4$} $$
since the diagonals of the hypercohomology spectral sequence are all trivial for $n\geq4$.
\end{proof}

Let $\varphi:X_{L}\rightarrow \bar{X}_{L}$ be the open embedding. By Axiom (7), we have an open/closed decoposition
$$\varphi:X_{L}\rightarrow \bar{X}_{L}\leftarrow \coprod_{v\in X_{\infty}}B_{W_{k(v)}}:i_{\infty}$$
In particular we have adjoint functors $\varphi_{!},\,\varphi^*,\,\varphi_{*}$. If we consider the induced functors on abelian sheaves, we have adjoint functors $i_{\infty}^*,\,i_{\infty*},\,i_{\infty}^!$, showing that $i_{\infty*}$ is exact (on abelian objects).
We obtain an exact sequence of sheaves
\begin{equation}\label{exact-sequence-WEsheaves}
0\rightarrow\varphi_!\varphi^*\mathcal{A}\rightarrow\mathcal{A}\rightarrow\prod_{v\mid\infty}i_{v*}i^*_{v}\mathcal{A}\rightarrow0.
\end{equation}
for any abelian object $\mathcal{A}$ of $\bar{X}_{L}$.
\begin{thm}\label{thm-coho-tau-cpct-support-RgZ}
One has canonically
\begin{eqnarray*}
\mathbb{H}^n(\bar{X}_{et},\tau_{\leq 2}R\gamma_*(\varphi_!\mathbb{Z}))
&=&0\mbox{ for $n=0$}\\
&=&(\prod_{X_{\infty}}\mathbb{Z})/\mathbb{Z} \mbox{ for $n=1$}\\
&=&Pic^1(\bar{X})^D \mbox{ for $n=2$}\\
&=&\mu_F^D \mbox{ for $n=3$}\\
&=&0 \mbox{ for $n\geq4$}\\
\end{eqnarray*}
\end{thm}
\begin{proof} Let $v\in X_{\infty}$. By Axiom (7), one has $\gamma\circ i_v=u_v\circ \alpha_v$, where $\alpha_v:B_{W_{k(v)}}\rightarrow\underline{Sets}$ is the unique map. The morphisms $i_v$ and $u_v$  are both closed embeddings so that $i_{v*}$ and $u_{v*}$ are both exact, hence one has
$$R(\gamma_*)i_{v*}\mathbb{Z}=u_{v*}R(\alpha_{v*})\mathbb{Z}.$$
This complex is concentrated in degree 0 since $R^n(\alpha_{v*})\mathbb{Z}=H^n(B_{W_{k(v)}},\mathbb{Z})=0$ for any $n\geq1$, and one has $R^0(\gamma_*)i_{v*}\mathbb{Z}=u_{v*}\mathbb{Z}$.
Hence the theorem follows from Theorem \ref{thm-coho-tau-RgZ}, exact sequence (\ref{exact-sequence-WEsheaves}), and $H^*(\bar{X}_{et},u_{v*}\mathbb{Z})=H^*(\underline{Sets},\mathbb{Z})$.
\end{proof}

\begin{rem}
A complex quasi-isomorphic to $\tau_{\leq 2}R\gamma_*(\varphi_!\mathbb{Z})$ was constructed in \cite{On the WE} using a more complicated method.
\end{rem}

\subsection{Dedekind zeta functions at $s=0$}
We denote by $(\oplus_{v\mid\infty}W_{k(v)})^1$ the kernel of the canonical morphism
$\oplus_{v\mid\infty}W_{k(v)}\rightarrow \mathbb{R}_{>0}$.
\begin{thm}
We have canonical isomorphisms :
\begin{eqnarray*}
H^n(\bar{X}_{L},\varphi_!\widetilde{\mathbb{R}})&=&(\prod_{v\mid\infty}\mathbb{R})/\mathbb{R}\mbox{ for $n=1$}\\
&=&Hom_c((\oplus_{v\mid\infty}W_{k(v)})^1,\mathbb{R}) \mbox{ for $n=2$}\\
&=&0\mbox{ for $n\neq1,2$}
\end{eqnarray*}
\end{thm}
\begin{proof}
The direct image $i_{v*}$ is exact hence the group
$H^n(\bar{X}_{L},\prod_{v\mid\infty}i_{v*}\widetilde{\mathbb{R}})$ is canonically isomorphic to
$$\prod_{v\mid\infty}H^n(B_{W_{k(v)}}\widetilde{\mathbb{R}})=
\prod_{v\mid\infty}\mathbb{R},\,Hom_c(\sum_{v\mid\infty}W_{k(v)},\mathbb{R}),\mbox{ and }0$$
for $n=0,1$ and $n\geq2$ respectively. Using the exact sequence (\ref{exact-sequence-WEsheaves}), the result for $n\geq3$ follows from Axiom (9). By (\ref{exact-sequence-WEsheaves}) we have the exact sequence
$$0\rightarrow H^0(\bar{X}_{L},\varphi_!\widetilde{\mathbb{R}})\rightarrow H^0(\bar{X}_{W},\widetilde{\mathbb{R}})=\mathbb{R}\rightarrow \prod_{v\mid\infty}H^n(B_{W_{k(v)}}\widetilde{\mathbb{R}})=
\prod_{v\mid\infty}\mathbb{R}\rightarrow H^1(\bar{X}_{L},\varphi_!\widetilde{\mathbb{R}})\rightarrow0$$
where the central map is the diagonal embedding. The result follows for $n=0,1$. For $n=2$, we have the exact sequence
$$Hom_c(Pic(\bar{X}),\mathbb{R})\rightarrow Hom_c(\sum_{v\mid\infty}W_{k(v)},\mathbb{R})\rightarrow
H^2(\bar{X}_{L},\varphi_!\widetilde{\mathbb{R}})\rightarrow 1$$
where, by Axiom (8), the first map is induced by the canonical morphism $\sum_{v\mid\infty} W_{k(v)}\rightarrow Pic(\bar{X})$. We obtain a canonical isomorphism
$$H^1(\bar{X}_{L},\varphi_!\widetilde{\mathbb{R}})=Hom_c((\sum_{v\mid\infty}W_{k(v)})^1,\mathbb{R}).$$
\end{proof}

\begin{defn}
We define the \emph{fundamental class} $\theta\in H^1(\bar{X}_{L},\widetilde{\mathbb{R}})$ as the canonical morphism
$$\theta\in H^1(\bar{X}_{L},\widetilde{\mathbb{R}})= Hom_{c}(Pic(\bar{X}),\mathbb{R})$$
\end{defn}
Recall that, for any closed point $v\in\bar{X}$, we have $H^n(B_{W_{k(v)}},\widetilde{\mathbb{R}})=\mathbb{R},\,Hom_c(W_{k(v)},\mathbb{R})\mbox{ and }0$ for $n=0,1$ and $n\geq2$ respectively.
\begin{defn}
For any closed point $v\in\bar{X}$, the \emph{$v$-fundamental class} is the canonical morphism $\theta_v:W_{k(v)}\rightarrow\mathbb{R}$:
$$\theta_v\in H^1(B_{W_{k(v)}},\widetilde{\mathbb{R}})=Hom_c(W_{k(v)},\mathbb{R}).$$
\end{defn}
The morphism obtained by cup product with $\theta_v$ is the canonical isomorphism
$$
\fonc{\cup\theta_v}{H^0(B_{W_{k(v)}},\widetilde{\mathbb{R}})=\mathbb{R}}
{H^1(B_{W_{k(v)}},\widetilde{\mathbb{R}})=Hom_c(W_{k(v)},\mathbb{R})}{1}{\theta_v}
$$

\begin{thm}\label{thm-cup-theta-explicit}
The morphism $\cup\theta$ obtained by cup product with the fundamental class $\theta$ is the canonical isomorphism :
$$
\fonc{\cup\theta}{H^1(\bar{X}_{L},\varphi_!\widetilde{\mathbb{R}})=(\prod_{v\mid\infty}\mathbb{R})/\mathbb{R}}
{H^2(\bar{X}_{L},\varphi_!\widetilde{\mathbb{R}})=Hom_c((\sum_{v\mid\infty}W_{k(v)})^1,\mathbb{R})}{v}{\theta_v\circ p_v}$$
where $p_v:(\sum_{w\mid\infty}W_{k(w)})^1\rightarrow W_{k(v)}$ is given by the projection.
\end{thm}
\begin{proof}
By Axiom (8), the morphism
$$
\appl{H^1(B_{Pic(\bar{X})},\widetilde{\mathbb{R}})=H^1(\bar{X}_{L},\widetilde{\mathbb{R}})}{H^1(\coprod_{v\in X_{\infty}}B_{W_{k(v)}},\widetilde{\mathbb{R}})= \prod_{v\mid\infty}H^1(B_{W_{k(v)}},\widetilde{\mathbb{R}})}{\theta}{(\theta_v)_{v\mid\infty}}
$$
which is induced by $\coprod_{v\in X_{\infty}}B_{W_{k(v)}}\rightarrow\bar{X}_{L}\rightarrow B_{Pic(\bar{X})}$, sends fundamental class to fundamental class. Hence the cup-product morphism $\cup\theta$ is induced by
$$(\cup\theta_v)_{v\mid\infty}:H^0(\bar{X}_{L},i_{\infty*}\widetilde{\mathbb{R}})=\prod_{v\mid\infty}{\mathbb{R}}
\longrightarrow H^1(\bar{X}_{L},i_{\infty*}\widetilde{\mathbb{R}})=\prod_{v\mid\infty} Hom_c(B_{W_{k(v)}},\widetilde{\mathbb{R}})$$
and the result follows. Note that $\cup\theta$ is well defined, since $\cup\theta(\sum_{v\mid\infty} v)=\sum_{v\mid\infty}\theta_v$ is the canonical map $\sum_{v\mid\infty}W_{k(v)}\rightarrow\mathbb{R}$, which vanishes on $(\sum_{v\mid\infty}W_{k(v)})^1$.
\end{proof}

\begin{thm}\label{thm-iso-Rn}
For any $n\geq1$, the morphism
$$R_n:\mathbb{H}^n(\bar{X}_{et},\tau_{\leq 2}R\gamma_*(\varphi_!\mathbb{Z}))\otimes\mathbb{R}\longrightarrow H^n(\bar{X}_{L},\varphi_!\widetilde{\mathbb{R}}),$$
induced by the morphism of sheaves $\varphi_!\mathbb{Z}\rightarrow\varphi_!\widetilde{\mathbb{R}}$, is an isomorphism.
\end{thm}
\begin{proof}
We denote by
$$\kappa_n:\mathbb{H}^n(\bar{X}_{et},\tau_{\leq 2}R\gamma_*(\varphi_!\mathbb{Z}))\longrightarrow H^1(\bar{X}_{L},\varphi_!\widetilde{\mathbb{R}})$$
the morphism induced by $\varphi_!\mathbb{Z}\rightarrow\varphi_!\widetilde{\mathbb{R}}$.
The result is obvious for $n\neq1,2$. The result is also clear for $n=1$, since $\kappa_1$ is canonically identified with
$$H^0(\coprod_{v\in X_{\infty}}B_{W_{k(v)}},\mathbb{Z})/H^0(\bar{X}_{L},\mathbb{Z})\rightarrow H^0(\coprod_{v\in X_{\infty}}B_{W_{k(v)}},\widetilde{\mathbb{R}})/H^0(\bar{X}_{L},\widetilde{\mathbb{R}})$$
hence $R_1$ is the identity. Assume that $n=2$. On the one hand, we have canonically
$$\mathbb{H}^2(\bar{X}_{et},\tau_{\leq 2}R\gamma_*(\varphi_!\mathbb{Z}))=H^2(\bar{X}_{L},\mathbb{Z})= H^1(\bar{X}_{L},\widetilde{\mathbb{S}}^1)/H^1(\bar{X}_{L},\widetilde{\mathbb{R}})$$
On the other hand, we have
$$H^2(\bar{X}_{L},\varphi_!\widetilde{\mathbb{R}})=
H^1(\coprod_{v\mid\infty}B_{W_{k(v)}},\widetilde{\mathbb{R}})/H^1(\bar{X}_{L},\widetilde{\mathbb{R}})=
H^1(\coprod_{v\mid\infty}B_{W_{k(v)}},\widetilde{\mathbb{S}}^1)/H^1(\bar{X}_{L},\widetilde{\mathbb{R}}).$$
and the map $R_2$ is induced by $$H^1(\bar{X}_{L},\widetilde{\mathbb{S}}^1)=Hom_c(Pic(\bar{X}),{\mathbb{S}}^1)\longrightarrow H^1(\coprod_{v\mid\infty}B_{W_{k(v)}},\widetilde{\mathbb{S}}^1)=Hom_c(\sum_{v\mid\infty}W_{k(v)},\mathbb{S}^1)$$
which is in turn induced by the canonical morphism $\sum_{v\mid\infty}W_{k(v)}\rightarrow Pic(\bar{X})$, as it follows from Axiom (8).
We have the exact sequence $$Hom_c(Pic(\bar{X}),\mathbb{R})\rightarrow Hom_c(Pic(\bar{X}),{\mathbb{S}}^1)\rightarrow Hom_c(Pic^1(\bar{X}),{\mathbb{S}}^1)\rightarrow 0$$ hence $\kappa_2$ is the morphism
$$\kappa_2:Hom_c(Pic^1(\bar{X}),{\mathbb{S}}^1)\longrightarrow Hom((\oplus_{v\mid\infty}W_{k(v)})^1,\mathbb{S}^1)=Hom((\oplus_{v\mid\infty}W_{k(v)})^1,\mathbb{R}),$$
where the first map is the Pontryagin dual of $(\oplus_{v\mid\infty}W_{k(v)})^1\rightarrow Pic^1(\bar{X})$.
Recall that we have the exact sequence of topological groups
$$0\rightarrow (\oplus_{v\mid\infty}W_{k(v)})^1/(\mathcal{O}_F^{\times}/\mu_F)\rightarrow Pic^1(\bar{X})\rightarrow Cl(F)\rightarrow0,$$
Then it is straightforward to check that there is a canonical identification
$$\mathbb{H}^2(\bar{X}_{et},\tau_{\leq 2}R\gamma_*(\varphi_!\mathbb{Z}))\otimes\mathbb{R}=Hom(\mathcal{O}_F^{\times},\mathbb{R})$$
and that the map $R_2$ is the morphism
$$R_2:Hom(\mathcal{O}_F^{\times},\mathbb{R})\longrightarrow Hom_c((\oplus_{v\mid\infty}W_{k(v)})^1,\mathbb{R})$$
which is the inverse of the isomorphism induced by the natural map
$\mathcal{O}_F^{\times}\rightarrow (\oplus_{v\mid\infty}W_{k(v)})^1$. In other words
$$R^{-1}_2:Hom_c((\oplus_{v\mid\infty}W_{k(v)})^1,\mathbb{R})\longrightarrow Hom(\mathcal{O}_F^{\times},\mathbb{R})$$
is induced by the natural map
$\mathcal{O}_F^{\times}\rightarrow (\oplus_{v\mid\infty}W_{k(v)})^1$.
\end{proof}
The morphisms $\cup\theta$, $R_1$ and $R_2$ have been made explicit during the proof of Theorem \ref{thm-cup-theta-explicit} and in Theorem \ref{thm-iso-Rn}. The following result is immediately checked.
\begin{cor}\label{cor-cuptheta}
We have a commutative diagram
\[ \xymatrix{
(\sum_{v\mid\infty}\mathbb{R})/\mathbb{R}\ar[r]^{D}\ar[d]^{R_1}&
Hom(\mathcal{O}_F^{\times},\mathbb{R})\ar[d]^{R_2}\\
(\sum_{v\mid\infty}\mathbb{R})/\mathbb{R}\ar[r]^{\cup\theta\,\,\,\,\,\,\,\,\,\,\,\,\,\,\,\,\,\,\,\,}
&Hom_c((\oplus_{v\mid\infty}W_{k(v)})^1,\mathbb{R})
}\]
where $D$ is the transpose of the usual regulator map $$\mathcal{O}_F^{\times}\otimes\mathbb{R}\longrightarrow (\sum_{v\mid\infty}\mathbb{R})^+.$$
\end{cor}

We denote by $\varphi:X_{L}\rightarrow \bar{X}_{L}$ the natural open embedding, and by $H_c^n(X_{L},\mathcal{A}):=H^n(\bar{X}_{L},\varphi_!\mathcal{A})$ the cohomology with compact support with coefficients in the abelian sheaf $\mathcal{A}$.
\begin{thm}\label{thm-Lichtenbaum-conj}\textbf{\emph{(Lichtenbaum's formalism)}}
Assume that $F$ is totally imaginary. Let $\bar{X}_{L}$ be any topos satisfying Axioms $(1)-(9)$ above. We denote by $\tau_{\leq2}R\gamma_*$ the truncated functor of the total derived functor $R\gamma_*$, where $\gamma$ is the morphism given by Property (1). Then the following is true:
\begin{itemize}
\item $\mathbb{H}^n(\bar{X}_{et},\tau_{\leq2}R\gamma_*(\varphi_!\mathbb{Z}))$ is finitely generated and zero for $n\geq4$.
\item The canonical map $$\mathbb{H}^n(\bar{X}_{et},\tau_{\leq2}R\gamma_*(\varphi_!\mathbb{Z}))\otimes\mathbb{R}\longrightarrow H_c^n(X_{L},\widetilde{\mathbb{R}})$$
is an isomorphism for any $n\geq0$.
\item There exists a fundamental class $\theta\in{H}^1(\bar{X}_{L},\widetilde{\mathbb{R}})$. The complex of finite dimensional vector spaces
    $$...\rightarrow{H}_c^{n-1}(X_{L},\widetilde{\mathbb{R}})\rightarrow {H}_c^n(X_{L},\widetilde{\mathbb{R}})\rightarrow {H}_c^{n+1}(X_{L},\widetilde{\mathbb{R}})\rightarrow...$$
    defined by cup product with $\theta$, is acyclic.
\item The vanishing order of the Dedekind zeta function $\zeta_F(s)$ at $s=0$ is given by
 $$\textsl{ord}_{s=0}\zeta_F(s)=\sum_{n\geq0}(-1)^n\,n\,\textsl{dim}_{\mathbb{R}}\,H_c^n(X_{L},\widetilde{\mathbb{R}})$$
\item The leading term coefficient $\zeta^*_F(s)$ at $s=0$ is given by the Lichtenbaum Euler characteristic :
$$\zeta^*_F(s)=\pm\prod_{n\geq0}|\mathbb{H}^n(\bar{X}_{et},\tau_{\leq2}R\gamma_*(\varphi_!\mathbb{Z}))_{\textsl{tors}}|^{(-1)^n}/
\textsl{det}(H_c^n(X_{L},\widetilde{\mathbb{R}}),\theta,B^*)$$
where $B^n$ is a basis of $\mathbb{H}^n(\bar{X}_{et},\tau_{\leq2}R\gamma_*(\varphi_!\mathbb{Z}))/\textsl{tors}$.
\end{itemize}
In particular those results hold for the Weil-\'etale topos $\bar{X}_W$ defined in Section 6.
\end{thm}
\begin{proof}This follows from Theorem \ref{thm-coho-tau-cpct-support-RgZ}, Theorem \ref{thm-cup-theta-explicit}, Theorem \ref{thm-iso-Rn}, Corollary \ref{cor-cuptheta} and from the analytic class number formula.
\end{proof}

\subsection{The sheaf $R^2\gamma_*\mathbb{Z}$}\label{sect-R2Z}

The étale sheaf $R^2\gamma_*\mathbb{Z}$ is the sheaf
associated to the presheaf
$$
\fonc{\mathcal{P}^2\mathbb{Z}}{Et_{\bar{X}}}{\underline{Ab}}{\bar{U}}
{(C^1_{\bar{U}})^{D} }.$$
Recall that if $\bar{U}$ is connected of function field $K(\bar{U})$, then
$C_{\bar{U}}$ is the $S$-idèle class group of $K(\bar{U})$, where
$S$ is the set of places of $K(\bar{U})$ not corresponding to a
point of $\bar{U}$. In other words, if one sets
$K=K(U)$ then $C_{\bar{U}}=C_{K,S}$ is the $S$-idele class group of $K$ defined by the exact sequence
$$\prod_{v\in U}\mathcal{O}_{K_v}^{\times}\rightarrow C_K\rightarrow C_{K,S}\rightarrow0.$$
The compact group $C^1_{\bar{U}}$
is then defined as the kernel of the canonical map $C_{\bar{U}}\rightarrow\mathbb{R}^{\times}$.
Note that such the finite set $S$ does not
necessarily contain all the archimedean places. The restrictions maps
of the presheaf $\mathcal{P}^2\mathbb{Z}$ are induced by the
canonical maps $C_{\bar{V}}\rightarrow C_{\bar{U}}$ (well defined
for any étale map $\bar{V}\rightarrow \bar{U}$ of connected étale
$\bar{X}$-schemes). By class field theory, one has a
covariantly functorial exact sequence of compact topological groups
$$0\rightarrow D^1_{\bar{U}}\rightarrow C^1_{\bar{U}}\rightarrow\pi_1(\bar{U}_{et})^{ab}\rightarrow0$$
where $\pi_1(\bar{U}_{et})^{ab}$ is the abelian étale fundamental group
of $\bar{U}$ and $D^1_{\bar{U}}$ is the connected component of $1$
in $C^1_{\bar{U}}$. Here $\pi_1(\bar{U})^{ab}$ is defined as the
abelianization of the profinite fundamental group of the
Artin-Verdier étale topos
$\bar{X}_{et}/{y\bar{U}}\simeq\bar{U}_{et}$. If we denote the
function field of $\bar{U}$ by $K(\bar{U})$ then this group is just
the Galois group of the maximal abelian extension of $K(\bar{U})$
unramified at every place of $K(\bar{U})$ corresponding to a point
of $\bar{U}$.

By Pontryagin duality, we obtain a contravariantly functorial
exact sequence of discrete abelian groups
\begin{equation}\label{exact-sequence-presheaves-R2Z}
0\rightarrow \pi_1^{ab}(\bar{U}_{et})^{D}\rightarrow
(C^1_{\bar{U}})^{D}\rightarrow
(D^1_{\bar{U}})^{D}\rightarrow0,
\end{equation}
i.e. an exact sequence of abelian étale presheaves on $\bar{X}$. On
the one hand, the sheaf associated to the presheaf
$$
\appl{Et_{\bar{X}}}{\underline{Ab}}{\bar{U}}
{\pi_1^{ab}(\bar{U}_{et})^{D}=H^2(\bar{U}_{et},\mathbb{Z})}$$
vanishes and the associated sheaf functor is exact on the other.
Therefore, the exact sequence (\ref{exact-sequence-presheaves-R2Z})
shows that $R^2\mathbb{Z}$ is the sheaf associated to the presheaf
$$
\fonc{P^2\mathbb{Z}}{Et_{\bar{X}}}{\underline{Ab}}{\bar{U}}
{(D^1_{\bar{U}})^{D} }.$$
The structure of the connected component $D^1_{\bar{U}}$ of the S-id\`ele class group $C^1_{\bar{U}}$ is not known in general.

\subsubsection{}
We consider the following open subscheme of $\bar{X}$
$$Y:=(X,X(\mathbb{R})).$$
Let $\bar{U}\rightarrow Y$ be a connected \'etale $Y$-scheme with function field $K:=K(\bar{U})$. Note that $U_{\infty}$ contains only real places. If $v$ is a real place of $K$, then we denote by $\mathcal{O}_{K_v}^{\times}=\pm1$ the kernel of the valuation ${K_v}^{\times}\rightarrow\mathbb{R}_{>0}$. The Leray spectral sequence associated to the morphism $Spec(K)_{et}\rightarrow \bar{U}_{et}$ gives an exact sequence
\begin{equation}\label{exact-sequence-for-etalecohomology}
0\rightarrow H^2(\bar{U}_{et},\mathbb{Z})\rightarrow H^2(G_K,\mathbb{Z})\rightarrow \sum_{v\in\bar{U}^0} H^2(I_v,\mathbb{Z})\rightarrow H^3(\bar{U}_{et},\mathbb{Z})\rightarrow H^3(G_K,\mathbb{Z})=0
\end{equation}
which is functorial in $\bar{U}$, with respect to the natural morphisms between Galois groups.  We have the following canonical identifications $$H^2(\bar{U}_{et},\mathbb{Z})=\pi_1(\bar{U}_{et})^D,\,\,\,H^2(G_K,\mathbb{Z})=G_K^D,\,\,\,H^2(I_v,\mathbb{Z})=I_v^D.$$ Then the central map of the exact sequence (\ref{exact-sequence-for-etalecohomology}) is induced by the natural maps $I_v^{ab}\rightarrow G_K^{ab}$, and the first map is given by the natural surjection $G_K^{ab}\rightarrow\pi_1(\bar{U}_{et})^{ab}$. Then global and local class field theory give the exact sequence
$$0\rightarrow\pi_1(\bar{U}_{et})^D\rightarrow (C_K/D_K)^D\rightarrow \sum_{v\in\bar{U}^0}(\mathcal{O}_{K_v}^{\times})^D\rightarrow H^3(\bar{U}_{et},\mathbb{Z})\rightarrow0$$
Here the functoriality is given by the norm maps, and $(C_K/D_K)^D\rightarrow \sum_{v\in\bar{U}^0}(\mathcal{O}_{K_v}^{\times})^D$ is the dual of the canonical map
$$\rho_{\bar{U}}:\prod_{v\in\bar{U}^0}\mathcal{O}_{K_v}^{\times}\longrightarrow C_K/D_K.$$ We obtain an isomorphism
$$H^3(\bar{U}_{et},\mathbb{Z})\simeq Ker(\rho_{\bar{U}})^D$$
which is functorial with respect to the presheaf structure on $H^3(-,\mathbb{Z})$ and with the norm maps on the right hand side.
Note that $$Ker(\rho_{\bar{U}})=\prod_{v\in\bar{U}^0}\mathcal{O}_{K_v}^{\times}\cap D_K=\prod_{v\in\bar{U}^0}\mathcal{O}_{K_v}^{\times}\cap D^1_K\subset C_K$$

\subsubsection{}For any connected \'etale scheme $\bar{U}$ over $Y=(X,X(\mathbb{R}))$, we have a functorial exact sequence of compact abelian topological groups
$$0\rightarrow \prod_{v\in \bar{U}^0}\mathcal{O}_{K(\bar{U})_v}^{\times}\cap D^1_{K(\bar{U})}\rightarrow D^1_{K(\bar{U})}\rightarrow D^1_{\bar{U}}\rightarrow 0$$
hence a contravariantly functorial exact sequence of discrete abelian groups
$$0\rightarrow (D^1_{\bar{U}})^D\rightarrow (D^1_{K(\bar{U})})^D\rightarrow (\prod_{v\in\bar{U}^0}\mathcal{O}_{K(\bar{U})_v}^{\times}\cap D^1_{K(\bar{U})})^D=H^3(\bar{U}_{et},\mathbb{Z}) \rightarrow 0$$
But the sheaf associated to the presheaf
$$
\appl{Et_{Y}}{\underline{Ab}}{\bar{U}}
{H^3(\bar{U}_{et},\mathbb{Z})}$$
vanishes. It follows that the sheaf $R^2\gamma_*\mathbb{Z}$ restricted to $Y_{et}$ is the sheaf associated to the presheaf
$$
\fonc{P}{Et_{Y}}{\underline{Ab}}{\bar{U}}
{(D^1_{K(\bar{U})})^{D} }$$
Let $\xi:Spec(F)_{et}\rightarrow Y_{et}$ be the morphism induced by the inclusion of the generic point. Let $\mathcal{F}$ the presheaf on $Et_{Spec(F)}$ sending a finite extension $K/F$ to $(D^1_{K})^D$.
By Tate's theorem, one has a functorial isomorphism of compact groups
$$D^1_{K}=(\mathbb{V}\otimes_{\mathbb{Z}}\mathcal{O}_K^{\times})\times(\prod_{r_2(K)}\mathbb{S}^1)$$
where $\mathbb{V}=\mathbb{Q}^D$ is the solenoid and $r_2(K)$ is the set of complex places of $K$. We obtain a functorial isomorphism of abelian groups
$$(D^1_{K})^D\simeq Hom(\mathcal{O}_K^{\times},\mathbb{Q})\oplus(\oplus_{r_2(K)}\mathbb{Z}).$$
It follows that $(D^1_{K})^D$ satisfies Galois descent hence $\mathcal{F}$ is a sheaf on the \'etale site of $Spec(F)$, and $P=\xi_*\mathcal{F}$ is a sheaf on $Et_Y$. We obtain the following description of $R^2\gamma_*\mathbb{Z}|Y$.
\begin{prop}\label{prop-R2Z-restricted-toY}
For any $\bar{U}$ connected \'etale over $Y=(X,X(\mathbb{R}))$, one has
$$R^2\gamma_*\mathbb{Z}(\bar{U})=(D^1_{K(\bar{U})})^{D}\simeq Hom(\mathcal{O}_{K(\bar{U})}^{\times},\mathbb{Q})\oplus(\sum_{r_2(K(\bar{U}))}\mathbb{Z}).$$
\end{prop}
Let $\phi:Y_{et}\rightarrow\bar{X}_{et}$ be the open immersion. Consider the adjunction map
\begin{equation}\label{une-ptite-adjunction}
R^2\gamma_*\mathbb{Z}\longrightarrow \phi_*\phi^*R^2\gamma_*\mathbb{Z}.
\end{equation}
For any $\bar{U}$ connected and \'etale over $\bar{X}$, the map $D^1_{K(\bar{U})}\rightarrow D^1_{\bar{U}}$ is surjective hence the induced map
$${P}^2\mathbb{Z}(\bar{U})=(D^1_{\bar{U}})^{D}\longrightarrow\phi_*\phi^*R^2\gamma_*\mathbb{Z}(\bar{U})
=R^2\gamma_*\mathbb{Z}(\bar{U}\times_{\bar{X}}Y)=(D^1_{K(\bar{U})})^{D}$$ is injective. Applying the (exact) associated functor, we see that the adjunction map (\ref{une-ptite-adjunction})
is also injective.

Recall that $R^2\gamma_*\mathbb{Z}$ is the sheaf associated to the presheaf on $Et_{\bar{X}}$ defined by
$\mathcal{P}^2\mathbb{Z}(\bar{U})=(C^1_{\bar{U}})^{D}$, and consider the presheaf
$$
\fonc{\phi_p\phi^p\mathcal{P}^2\mathbb{Z}}{Et_{\bar{X}}}{\underline{Ab}}{\bar{U}}
{(C^1_{\bar{U}\times_{\bar{X}}Y})^{D}}$$
For any $\bar{U}$ connected and \'etale over $\bar{X}$ and such that $\bar{U}$ does not contain all the places of $K(\bar{U})$, we have an exact sequence
$$0\rightarrow \prod_{U_{\infty}-U(\mathbb{R})}\mathbb{S}^1\rightarrow C^1_{\bar{U}\times_{\bar{X}}Y}\rightarrow  C^1_{\bar{U}}\rightarrow0 $$
inducing an exact sequence of discrete abelian groups
$$0\rightarrow (C^1_{\bar{U}})^D\rightarrow (C^1_{\bar{U}\times_{\bar{X}}Y})^D\rightarrow\prod_{U_{\infty}-U(\mathbb{R})}\mathbb{Z}\rightarrow0.$$
In other words, we have an exact sequence of presheaves on ${Et}_{\bar{X}}'$
$$0\rightarrow\mathcal{P}^2\mathbb{Z}\rightarrow \phi_p\phi^p\mathcal{P}^2\mathbb{Z}\rightarrow\prod_{v\in\bar{X}-Y}u_{v*}\mathbb{Z}\rightarrow0$$
where $Et_{\bar{X}}'$ is the full subcategory of $Et_{\bar{X}}$ consisting of connected objects $\bar{U}$ such that $\bar{U}$ does not contain all the places of $K(\bar{U})$, and the adjoint functors $\phi^p$ and $\phi_p$ are functors between categories of presheaves. But $Et_{\bar{X}}'$ is a topologically generating full subcategory of the \'etale site ${Et}_{\bar{X}}$. Applying the associated sheaf functor, we get an exact sequence of sheaves
\begin{equation}\label{exact-sequ-for-R2Z}
0\rightarrow R^2\gamma_*\mathbb{Z}\rightarrow \phi_*\phi^*R^2\gamma_*\mathbb{Z}\rightarrow\prod_{v\in\bar{X}-Y}u_{v*}\mathbb{Z}\rightarrow0
\end{equation}
since the sheaf associated to $\phi_p\phi^p\mathcal{P}^2\mathbb{Z}$ is just $\phi_*\phi^*R^2\gamma_*\mathbb{Z}$. In order to check this last claim, we consider the open-closed decomposition
$$\phi: Y_{et}\longrightarrow \bar{X}_{et}\longleftarrow \coprod_{v\in \bar{X}-Y}\underline{Sets}:u $$
where the gluing functor $u^*\phi_*$ sends a sheaf $\mathcal{F}$ on $Y_{et}$ to the collection of the stalks $(\mathcal{F}_v)_{v\in\bar{X}-Y}$ (here $\mathcal{F}_v$ is the stalk of $\mathcal{F}$ at the geometric point $v:Spec(\mathbb{C})\rightarrow Y$). It follows easily that
$a(\phi_pP)=\phi_*a(P)$ for any presheaf $P$ on $Y$, where $a$ denotes the associated sheaf functor. Hence we have
$$a(\phi_p\phi^p\mathcal{P}^2\mathbb{Z})\simeq \phi_*a(\phi^p\mathcal{P}^2\mathbb{Z})\simeq
\phi_*\phi^*a(\mathcal{P}^2\mathbb{Z})\simeq \phi_*\phi^*R^2\gamma_*\mathbb{Z}.$$
In view of Proposition \ref{prop-R2Z-restricted-toY}, we obtain the following result, where $r_2(K(\bar{U}))-U_{\infty}$ denotes the set of complex places of $K(\bar{U})$ which do not correspond to a point of $\bar{U}$.
\begin{thm}
For any connected \'etale $\bar{X}$-scheme $\bar{U}$, one has
\begin{eqnarray*}
R^2\gamma_*\mathbb{Z}(\bar{U})&=&(D^1_{K(\bar{U})}/\prod_{U_{\infty}-U(\mathbb{R})}\mathbb{S}^1)^{D}\\
&\simeq& Hom(\mathcal{O}_{K(\bar{U})}^{\times},\mathbb{Q})\oplus\sum_{r_2(K(\bar{U}))-U_{\infty}}\mathbb{Z}.
\end{eqnarray*}
\end{thm}

\begin{cor}
Assume that $F$ is totally imaginary. Then the \'etale sheaf $R^2\gamma_*\mathbb{Z}$ is acyclic for the global sections functor.
\end{cor}
\begin{proof}
In view of the exact sequence (\ref{exact-sequ-for-R2Z}), it is enough to show that $\phi_*\phi^*R^2\gamma_*\mathbb{Z}$ is acyclic for the global sections functor. Here we denote by $\xi:Spec(F)_{et}\rightarrow\bar{X}_{et}$ the morphism induced by the inclusion of the generic point, and by $\mathcal{F}$ the presheaf on $Et_{Spec(F)}$ sending a finite extension $K/F$ to $(D^1_{K})^D$. Then we have
$\phi_*\phi^*R^2\gamma_*\mathbb{Z}=\xi_*\mathcal{F}$. But for any finite Galois extension $K/F$ of group $G$, the $G$-module $(D^1_{K})^D$ is the product of a $\mathbb{Q}$-vector space by an induced $G$-module. It follows that the sheaf $\xi_*\mathcal{F}$ is acyclic for the global sections functor.
\end{proof}

\end{document}